\documentclass[12pt,leqno]{article}
\usepackage{amsmath,amssymb,bbm,array}
\usepackage{amsthm}
\usepackage{titlesec} % 重設章節標題所需的巨集
\usepackage[all,cmtip]{xy}

\titlespacing{\section}{0cm}{3.5pc}{1.5pc}

\makeatletter
\def\@citex[#1]#2{\if@filesw\immediate\write\@auxout{\string\citation{#2}}\fi
  \def\@citea{}\@cite{\@for\@citeb:=#2\do
    {\@citea\def\@citea{\@citesep}\@ifundefined
       {b@\@citeb}{{\bf ?}\@warning
       {Citation `\@citeb' on page \thepage \space undefined}}%
{\csname b@\@citeb\endcsname}}}{#1}}
\def\@citesep{; }
\makeatother

\newtheoremstyle{Kang}{}{}{\itshape}{}{\bf}{}{.5em}{}
\theoremstyle{Kang}
\newtheorem{theorem}{Theorem}[section]

\newtheorem{lemma}[theorem]{Lemma}

\newtheoremstyle{Kremark}{}{}{}{}{\bf}{}{.5em}{}
\theoremstyle{Kremark}
\newtheorem{defn}[theorem]{Definition}

\newtheorem{other}{}
\newenvironment{idef}[1]{\renewcommand{\theother}{#1}\begin{other}}{\end{other}}
\newenvironment{Case}[1]{\medskip {\it Case #1.}}{}

\allowdisplaybreaks[1]  % 允許多行數學式分頁
\numberwithin{equation}{section}

\def\bm#1{\mathbbm{#1}}
\def\fn#1{\mathop{{\rm #1}\vphantom{\sin}}} % function work like \sin
\def\bm#1{\mathbbm{#1}}

\title{Noether's problem for the groups with \\ a cyclic subgroup of index $4$}
\author{Ming-chang Kang$^{a,1}$, Ivo M. Michailov$^{b,2}$ and Jian Zhou$^{c,1,3}$ \\[3mm]
\begin{minipage}{16cm} \begin{description} \itemsep=-1pt
\item[] $^{a}$Department of Mathematics and Taida Institute of Mathematical\\ Sciences,
National Taiwan University, Taipei
\item[] $^{b}$Faculty of Mathematics and Informatics, Constantin Preslavski\\ University, Shoumen, Bulgaria
\item[] $^{c}$School of Mathematical Sciences, Peking University, Beijing
\end{description} \end{minipage}}
\date{}

\begin{document}

\maketitle

\footnote{\hspace*{-7.5mm}
E-mail addresses: kang@math.ntu.edu.tw,
ivo\_\,michailov@yahoo.com, zhjn@math.pku.edu.cn}
\footnote{\hspace*{-6mm}$^1\,$Both authors were partially
supported by National Center for Theoretic Sciences (Taipei
Office).} \footnote{\hspace*{-6mm}$^2\,$ Partially supported by
project No. RD-05-156/25.02.2011 of Shumen University.}
\footnote{\hspace*{-6mm}$^3\,$ The work of this paper was finished
when the third-named author visited National Taiwan University
under the support by National Center for Theoretic Sciences
(Taipei Office).}

\begin{abstract}
{\bf Abstract.} Let $G$ be a finite group and $k$ be a field. Let
$G$ act on the rational function field $k(x_g:g\in G)$ by
$k$-automorphisms defined by $g\cdot x_h=x_{gh}$ for any $g,h\in
G$. Noether's problem asks whether the fixed field
$k(G)=k(x_g:g\in G)^G$ is rational (i.e. purely transcendental)
over $k$. Theorem 1. If $G$ is a group of order $2^n$ ($n\ge 4$)
and of exponent $2^e$ such that (i) $e\ge n-2$ and (ii)
$\zeta_{2^{e-1}} \in k$, then $k(G)$ is $k$-rational. Theorem 2.
Let $G$ be a group of order $4n$ where $n$ is any positive integer
(it is unnecessary to assume that $n$ is a power of $2$). Assume
that {\rm (i)} $\fn{char}k \ne 2$, $\zeta_n \in k$, and {\rm (ii)}
$G$ contains an element of order $n$. Then $k(G)$ is rational over
$k$, except for the case $n=2m$ and $G \simeq C_m \rtimes C_8$
where $m$ is an odd integer and the center of $G$ is of even order
(note that $C_m$ is normal in $C_m \rtimes C_8$) ; for the
exceptional case, $k(G)$ is rational over $k$ if and only if at
least one of $-1, 2, -2$ belongs to $(k^{\times})^2$.
\end{abstract}

\bigskip

%------------------------S1
\section{Introduction}

Let $k$ be any field and $G$ be a finite group. Let $G$ act on the
rational function field $k(x_g:g\in G)$ by $k$-automorphisms such
that $g\cdot x_h=x_{gh}$ for any $g,h\in G$. Denote by $k(G)$ the
fixed field $k(x_g:g\in G)^G$. Noether's problem asks whether
$k(G)$ is rational (=purely transcendental) over $k$. It is
related to the inverse Galois problem, to the existence of generic
$G$-Galois extensions over $k$, and to the existence of versal
$G$-torsors over $k$-rational field extensions \cite[33.1,
p.86]{Sw,Sa,GMS}. Noether's problem for abelian groups was studied
extensively by Swan, Voskresenskii, Endo, Miyata and Lenstra, etc.
The reader is referred to Swan's paper for a survey of this
problem \cite{Sw}.

On the other hand, just a handful of results about Noether's
problem are obtained when the groups are not abelian. It is the
case even when the group $G$ is a $p$-group. The reader is
referred to \cite{CK,Ka6,HuK,Ka4} for previous results of
Noether's problem for $p$-groups. In the following we will list
only those results pertaining to the 2-groups.

%-------------------------t1.1
\begin{theorem}[Chu, Hu and Kang \cite{CHK,Ka2}] \label{t1.1}
Let $k$ be any field.
Suppose that $G$ is a non-abelian group of order $8$ or $16$.
Then $k(G)$ is rational over $k$,
except when $\fn{char} k\ne 2$ and $G=Q_{16}$,
the generalized quaternion group of order $16$.
When $\fn{char}k\ne 2$ and $G=Q_{16}$,
then $k(G)$ is also rational over $k$ provided that $k(\zeta_8)$ is a cyclic extension over $k$
where $\zeta_8$ is a primitive $8$-th root of unity.
\end{theorem}

%---------------------------t1.2
\begin{theorem}[Serre {\cite[Theorem 34.7]{GMS}}] \label{t1.2}
If $G=Q_{16}$, then $\bm{Q}(G)$ is not stably rational over $\bm{Q}$;
in particular, it is not rational over $\bm{Q}$.
\end{theorem}

We don't know the answer whether $k(G)$ is rational over $k$ or
not, if $G=Q_{16}$ and $k$ is any field other than $\bm{Q}$ such
that $k(\zeta_8)$ is not a cyclic extension of $k$. The reader is
referred to \cite{CHKP,CHKK} for groups of order $32$ and $64$.

Among the known results of Noether's problem for non-abelian
$p$-groups, except the situations in Theorem \ref{t1.1},
assumptions on the existence of ``enough" roots of unity always
arose (see, for example, the following Theorem \ref{t1.3}). In
fact, even when $G$ is a non-abelian $p$-group of order $p^3$
where $p$ is an odd prime number, it is not known how to find a
necessary and sufficient condition such as $\bm{Q}(G)$ is rational
over $\bm{Q}$ (see \cite{Ka5}). Similarly, without assuming the
existence of roots of unity, we don't have a good criterion to
guarantee $\bm{Q}(G)$ is rational where $G$ is a non-abelian group
of order $32$ (compare with the results in \cite{CHKP}). Thus it
will be desirable if we can weaken the assumptions on the
existence of roots of unity.

Now we turn to metacyclic $p$-groups. For results of a general
metacyclic $p$-group, see \cite{Ka6}. The following theorem
provides a sharper result for metacyclic $p$-groups with a large
cyclic subgroup. We state the result for the $2$-groups only.

%----------------------------t1.3
\begin{theorem}[Hu and Kang \cite{HuK,Ka4}] \label{t1.3}
Let $n\ge 4$ and $G$ be a non-abelian group of order $2^n$.
Assume that either {\rm (i)} $\fn{char}k=2$,
or {\rm (ii)} $\fn{char}k\ne 2$ and $k$ contains a primitive $2^{n-2}$-th root of unity.
If $G$ contains an element whose order $\ge 2^{n-2}$,
then $k(G)$ is rational over $k$.
\end{theorem}

The purpose of this paper is to generalize the above Theorem
\ref{t1.3} in two directions. The first main result is the
following theorem.

%-------------------------t1.4
\begin{theorem} \label{t1.4}
Let $n\ge 4$ and $G$ be a group of order $2^n$ and of exponent $2^e$ where $e\ge n-2$.
Assume that either {\rm (i)} $\fn{char}k=2$,
or {\rm (ii)} $\fn{char}k \ne 2$ and $k$ contains a primitive $2^{e-1}$-th root of unity.
Then $k(G)$ is rational over $k$.
\end{theorem}

For the second generalization of Theorem \ref{t1.3}, recall a
theorem of Michailov \cite{Mi}. Imitating the definition of the
$2$-groups modular groups, dihedral groups, quasi-dihedral groups,
generalized quaternion groups (see \cite[Theorem 1.9]{HuK}), the
groups $M_{8n}, D_{8n}$, $SD_{8n}, Q_{8n}$ are defined in
\cite[Section 1]{Mi} where $n$ is any positive integer and it is
unnecessary to assume that $n$ is a power of $2$. Denote by $H$
anyone of these groups. Note that $H$ is a group of order $8n$
containing an element of order $4n$. The following theorem
generalizes Theorem \ref{t1.3} in another direction.

%-------------------------t1.7
\begin{theorem}[Michailov {\cite[Theorem 1.2]{Mi}}] \label{t1.7}
Let $H$ be a group of order $8n$ defined above, and $G$ be a
central extension of $H$ defined by $1 \rightarrow C_2 \rightarrow
G \rightarrow H \rightarrow 1$ where $C_2$ is the cyclic group of
order two. If k is a field containing a primitive $4n$-th root of
unity. Then $k(G)$ is rational over $k$.
\end{theorem}

The second main result of this paper is the following theorem
which generalizes Theorem \ref{t1.3} and Theorem \ref{t1.7}.

%-------------------------t1.8
\begin{theorem} \label{t1.8}
Let $G$ be a group of order $4n$ where $n$ is any positive integer
{\rm(}it is unnecessary to assume that $n$ is a power of $2
\,${\rm)}. Assume that {\rm (i)} $\fn{char}k \ne 2$ and $k$
contains a primitive $n$-th root of unity, and {\rm (ii)} $G$
contains an element of order $n$. Then $k(G)$ is rational over
$k$, except for the case $n=2m$ and $G \simeq C_m \rtimes C_8$
where $m$ is an odd integer and the center of $G$ is of even order
{\rm(}note that $C_m$ is normal in $C_m \rtimes C_8${\rm)}; for
the exceptional case, $k(G)$ is rational over $k$ if and only if
at least one of $-1, 2, -2$ belongs to $(k^{\times})^2$.
\end{theorem}

The main tool in proving the above Theorem \ref{t1.8} is Theorem
\ref{t2.7} and Theorem \ref{t2.8} in Section 2 of this paper.
Theorem \ref{t2.7} and Theorem \ref{t2.8} are consequences of
Yamasaki's paper \cite{Ya} and the paper by Hoshi, Kitayama and
Yamasaki \cite{HKY}. Without these two papers, the proof of
Theorem \ref{t1.8} would be very involved.

The proof of Theorem \ref{t1.4} is more computational. It relies
on the classification of 2-groups with a cyclic subgroup of index
$4$ \cite{Ni}. Note that in order to prove Theorem \ref{t1.4} we
may assume the following extra conditions on $G$ and $k$ without
loss of generality
\begin{equation}
n\ge 5, ~ |G|=2^n, ~ \exp(G)=2^{n-2}, \text{ $G$ is non-abelian}, ~
\fn{char}k\ne 2 \text{ and }\zeta_{2^{n-3}}\in k. \label{eq1.1}
\end{equation}

For, it is not difficult to prove Theorem \ref{t1.4} when $G$ is an abelian group
by applying Lenstra's Theorem \cite{Le}.
Moreover, Kuniyoshi's Theorem asserts that,
if $\fn{char}k=p>0$ and $G$ is a $p$-group,
then $k(G)$ is rational over $k$ \cite[Corollary 1.2]{Ku,KP}.
Thus we may assume that $G$ is non-abelian and $\fn{char}k\ne 2$.
When $G$ is a non-abelian group of order $2^n$,
the case of Theorem \ref{t1.4} when $n=4$ is taken care by Theorem \ref{t1.1},
and the case when $\exp(G)=2^{n-1}$ is taken care by Theorem \ref{t1.3}.
Thus only the situation of \eqref{eq1.1} remains.

The key idea to prove Theorem \ref{t1.4} is, by applying Theorem
\ref{t2.2}, to find a low-dimensional faithful $G$-subspace
$W=\bigoplus_{1\le i\le m} k\cdot y_i$ of the regular
representation space $\bigoplus_{g\in G} k\cdot x(g)$ and to show
that $k(y_i: 1\le i\le m)^G$ is rational over $k$. The subspace
$W$ is obtained as an induced representation from some abelian
subgroup of $G$. This method is reminiscent of some techniques
exploited in \cite{Ka4}. However, the proof of Theorem \ref{t1.4}
is more subtle and requires elaboration. For examples, in
\cite{Ka4}, the following two theorems were used to solve the
rationality problem for many groups $G_i$ in Theorem \ref{t2.1}.

%----------------------------t1.5
\begin{theorem}[\cite{Ka6}] \label{t1.5}
Let $k$ be a field and $G$ be a metacyclic $p$-group.
Assume that {\rm (i)} $\fn{char}k=p>0$,
or {\rm (ii)} $\fn{char}k\ne p$ and $\zeta_e\in k$ where $e=\exp(G)$.
Then $k(G)$ is rational over $k$.
\end{theorem}

%----------------------------t1.6
\begin{theorem}[{\cite[Theorem 1.4]{Ka3}}] \label{t1.6}
Let $k$ be a field and $G$ be a finite group.
Assume that {\rm (i)} $G$ contains an abelian normal subgroup $H$ so that $G/H$ is cyclic of order $n$,
{\rm (ii)} $\bm{Z}[\zeta_n]$ is a unique factorization domain,
and {\rm (iii)} $\zeta_e \in k$ where $e$ is the exponent of $G$.
If $G\to GL(V)$ is any finite-dimensional linear representation of $G$ over $k$,
then $k(V)^G$ is rational over $k$.
\end{theorem}

Because we assume $\zeta_{2^{n-3}}\in k$ (instead of
$\zeta_{2^{n-2}}\in k$) in \eqref{eq1.1}, the above two theorems
are not directly applicable in the present situation. This is the
reason why we should find judiciously a faithful subspace $W$.
Fortunately we can find these subspaces $W$ in an almost unified
way. In fact, the proof for the group $G_8$ in Theorem \ref{t2.1}
is a typical case; the proof for other groups is either similar to
that of $G_8$ or has appeared in \cite{Ka4}.

We organize this paper as follows. In Section 2 we recall
Ninomiya's classification of non-abelian groups $G$ with $|G|=2^n$
and $\exp(G)=2^{n-2}$ (where $n\ge 4$). We also recall some
preliminaries which will be used in the proof of Theorem
\ref{t1.4} and Theorem \ref{t1.8}. The proof of Theorem \ref{t1.4}
is given in Section 3. In Section 4 the proof of Theorem
\ref{t1.8} is given. Note that Theorem \ref{t4.4} is of interest
itself.

\begin{idef}{Standing Notations.}
Throughout this article, $k(x_1,\ldots,x_n)$ or $k(x,y)$ will be
rational function fields over $k$. $\zeta_n$ denotes a primitive
$n$-th root of unity in some algebraic extension of the field $k$.
Whenever we write $\zeta_n \in k$, it is understood that either
$\fn{char}k=0$ or $\fn{char}k=p >0$ with $gcd \{p, n \}=1$.

A field extension $L$ of $k$ is called rational over $k$ (or
$k$-rational, for short) if $L\simeq k(x_1,\ldots,x_n)$ over $k$
for some integer $n$. $L$ is stably rational over $k$ if
$L(y_1,\ldots,y_m)$ is rational over $k$ for some $y_1,\ldots,y_m$
which are algebraically independent over $L$. Recall that, if $G$
is a finite group, $k(G)$ denotes $k(x_g:g\in G)^G$ where $h\cdot
x_g=x_{hg}$ for $h,g\in G$.
\end{idef}
The exponent of a finite group $G$, denoted by $\exp(G)$, is
$\fn{lcm}\{\fn{ord}(g):g\in G\}$ where $\fn{ord}(g)$ is the order
of $g$. The cyclic group of order $n$ is denoted by $C_n$.

If $G$ is a group acting on a rational function field
$k(x_1,\ldots,x_n)$ by $k$-automorphisms, the actions of $G$ are
called purely monomial actions if, for any $\sigma \in G$, any
$1\le j\le n$, $\sigma\cdot x_j=\prod_{1\le i\le n} x_i^{a_{ij}}$
where $a_{ij}\in\bm{Z}$; similarly, the actions of $G$ are called
monomial actions if, for any $\sigma \in G$, any $1\le j\le n$,
$\sigma\cdot x_j=\lambda_j(\sigma)\cdot\prod_{1\le i\le n}
x_i^{a_{ij}}$ where $a_{ij}\in\bm{Z}$ and $\lambda_j(\sigma)\in
k\backslash \{0\}$. When $G$ acts on $k(x_1,\ldots,x_n)$ by
monomial actions and $\alpha \in k\backslash \{0\}$, we say that
$G$ acts on $k(x_1,\ldots,x_n)$ by monomial actions with
coefficients in $\langle \alpha \rangle$ if $\lambda_j(\sigma)\in
\langle \alpha \rangle$ for all $\sigma \in G$, for all $1 \le j
\le n$ (here $\langle \alpha \rangle$ denotes the multiplicative
subgroup in $k\backslash \{0\}$ generated by $\alpha$). All the
groups in this article are finite groups.

%----------------------------------------------S2
\section{Preliminaries}

%------------------------t2.1
\begin{theorem}[Ninomiya {\cite[Theorem 2]{Ni}}] \label{t2.1}
Let $n\ge 4$.
The finite non-abelian groups of order $2^n$ which have a cyclic subgroup of index $4$,
but haven't a cyclic subgroup of index $2$ are of the following types:
\begin{enumerate}
\item[{\rm (I)}] $n\ge 4$ \\
$G_1=\langle\sigma,\tau:\sigma^{2^{n-2}}=\tau^4=1,\tau^{-1}\sigma\tau=\sigma^{1+2^{n-3}}\rangle$, \\
$G_2=\langle\sigma,\tau,\lambda: \sigma^{2^{n-2}}=\lambda^2=1,\sigma^{2^{n-3}}=\tau^2,
\tau^{-1}\sigma\tau=\sigma^{-1},\sigma\lambda=\lambda\sigma,\tau\lambda=\lambda\tau\rangle$, \\
$G_3=\langle\sigma,\tau,\lambda: \sigma^{2^{n-2}}=\tau^2=\lambda^2=1,
\tau^{-1}\sigma\tau=\sigma^{-1},\sigma\lambda=\lambda\sigma,\tau\lambda=\lambda\tau\rangle$, \\
$G_4=\langle\sigma,\tau,\lambda:\sigma^{2^{n-2}}=\tau^2=\lambda^2=1,\sigma\tau=\tau\sigma,
\sigma\lambda=\lambda\sigma,\lambda^{-1}\tau\lambda=\sigma^{2^{n-3}}\tau\rangle$, \\
$G_5=\langle \sigma,\tau,\lambda: \sigma^{2^{n-2}}=\tau^2=\lambda^2=1,\sigma\tau=\tau\sigma,
\lambda^{-1}\sigma\lambda=\sigma\tau,\tau\lambda=\lambda\tau\rangle$.
\item[{\rm (II)}] $n\ge 5$ \\
$\begin{array}{@{}l@{\;}c@{\;}l}
G_6 &=& \langle\sigma,\tau:\sigma^{2^{n-2}}=\tau^4=1,\tau^{-1}\sigma\tau=\sigma^{-1}\rangle, \\
G_7 &=& \langle\sigma,\tau:\sigma^{2^{n-2}}=\tau^4=1,\tau^{-1}\sigma\tau=\sigma^{-1+2^{n-3}}\rangle, \\
G_8 &=& \langle\sigma,\tau:\sigma^{2^{n-2}}=1,\sigma^{2^{n-3}}=\tau^4,\tau^{-1}\sigma\tau=\sigma^{-1}\rangle, \\
G_9 &=& \langle\sigma,\tau:\sigma^{2^{n-2}}=\tau^4=1,\sigma^{-1}\tau\sigma=\tau^{-1}\rangle, \\
G_{10} &=& \langle\sigma,\tau,\lambda:\sigma^{2^{n-2}}=\tau^2=\lambda^2=1,\tau^{-1}\sigma\tau=\sigma^{1+2^{n-3}},
\sigma\lambda=\lambda\sigma,\tau\lambda=\lambda\tau\rangle, \\
G_{11} &=& \langle\sigma,\tau,\lambda:\sigma^{2^{n-2}}=\tau^2=\lambda^2=1,\tau^{-1}\sigma\tau=\sigma^{-1+2^{n-3}},
\sigma\lambda=\lambda\sigma,\tau\lambda=\lambda\tau\rangle, \\
G_{12} &=& \langle\sigma,\tau,\lambda:\sigma^{2^{n-2}}=\tau^2=\lambda^2=1,\sigma\tau=\tau\sigma,
\lambda^{-1}\sigma\lambda=\sigma^{-1},\lambda^{-1}\tau\lambda=\sigma^{2^{n-3}}\tau\rangle, \\
G_{13} &=& \langle\sigma,\tau,\lambda:\sigma^{2^{n-2}}=\tau^2=\lambda^2=1,\sigma\tau=\tau\sigma,
\lambda^{-1}\sigma\lambda=\sigma^{-1}\tau,\tau\lambda=\lambda\tau\rangle, \\
G_{14} &=& \langle\sigma,\tau,\lambda:\sigma^{2^{n-2}}=\tau^2=1,\sigma^{2^{n-3}}=\lambda^2,\sigma\tau=\tau\sigma,
\lambda^{-1}\sigma\lambda=\sigma^{-1}\tau,\tau\lambda=\lambda\tau\rangle, \\
G_{15} &=& \langle\sigma,\tau,\lambda:\sigma^{2^{n-2}}=\tau^2=\lambda^2=1,\tau^{-1}\sigma\tau=\sigma^{1+2^{n-3}},
\lambda^{-1}\sigma\lambda=\sigma^{-1+2^{n-3}},\tau\lambda=\lambda\tau\rangle, \\
G_{16} &=& \langle\sigma,\tau,\lambda:\sigma^{2^{n-2}}=\tau^2=\lambda^2=1,\tau^{-1}\sigma\tau=\sigma^{1+2^{n-3}},
\lambda^{-1}\sigma\lambda=\sigma^{-1+2^{n-3}}, \\
&& ~\lambda^{-1}\tau\lambda=\sigma^{2^{n-3}}\tau\rangle, \\
G_{17} &=& \langle\sigma,\tau,\lambda:\sigma^{2^{n-2}}=\tau^2=\lambda^2=1,\tau^{-1}\sigma\tau=\sigma^{1+2^{n-3}},
\lambda^{-1}\sigma\lambda=\sigma\tau,\tau\lambda=\lambda\tau\rangle, \\
G_{18} &=& \langle\sigma,\tau,\lambda:\sigma^{2^{n-2}}=\tau^2=1,\lambda^2=\tau,
\tau^{-1}\sigma\tau=\sigma^{1+2^{n-3}},\lambda^{-1}\sigma\lambda=\sigma^{-1}\tau\rangle.
\end{array}$
\item[{\rm (III)}] $n\ge 6$ \\
$\begin{array}{@{}r@{\;}l}
G_{19}=& \langle \sigma,\tau:\sigma^{2^{n-2}}=\tau^4=1,\tau^{-1}\sigma\tau=\sigma^{1+2^{n-4}}\rangle, \\
G_{20}=& \langle \sigma,\tau:\sigma^{2^{n-2}}=\tau^4=1,\tau^{-1}\sigma\tau=\sigma^{-1+2^{n-4}}\rangle, \\
G_{21}=& \langle \sigma,\tau:\sigma^{2^{n-2}}=1,\sigma^{2^{n-3}}=\tau^4,\sigma^{-1}\tau\sigma=\tau^{-1}\rangle, \\
G_{22}=& \langle\sigma,\tau,\lambda:\sigma^{2^{n-2}}=\tau^2=\lambda^2=1,\sigma\tau=\tau\sigma,
\lambda^{-1}\sigma\lambda=\sigma^{1+2^{n-4}}\tau,\lambda^{-1}\tau\lambda=\sigma^{2^{n-3}}\tau\rangle, \\
G_{23}=& \langle\sigma,\tau,\lambda:\sigma^{2^{n-2}}=\tau^2=\lambda^2=1,\sigma\tau=\tau\sigma,
\lambda^{-1}\sigma\lambda=\sigma^{-1+2^{n-4}}\tau,\lambda^{-1}\tau\lambda=\sigma^{2^{n-3}}\tau\rangle, \\
G_{24}=& \langle\sigma,\tau,\lambda:\sigma^{2^{n-2}}=\tau^2=\lambda^2=1,\tau^{-1}\sigma\tau=\sigma^{1+2^{n-3}},
\lambda^{-1}\sigma\lambda=\sigma^{-1+2^{n-4}},\tau\lambda=\lambda\tau\rangle, \\
G_{25}=& \langle\sigma,\tau,\lambda:\sigma^{2^{n-2}}=\tau^2=1,\sigma^{2^{n-3}}=\lambda^2,
\tau^{-1}\sigma\tau=\sigma^{1+2^{n-3}},\lambda^{-1}\sigma\lambda=\sigma^{-1+2^{n-4}}, \\
& ~\tau\lambda=\lambda\tau\rangle,
\end{array}$
\item[{\rm (IV)}] $n=5$ \\
$G_{26}= \langle\sigma,\tau,\lambda:\sigma^8=\tau^2=1,\sigma^4=\lambda^2,\tau^{-1}\sigma\tau=\sigma^5,
\lambda^{-1}\sigma\lambda=\sigma\tau,\tau\lambda=\lambda\tau\rangle$.
\end{enumerate}
\end{theorem}

%----------------------------t2.2
\begin{theorem}[{\cite[Theorem 1]{HK}}] \label{t2.2}
Let $G$ be a finite group acting on $L(x_1,\ldots,x_n)$,
the rational function field of $n$ variables over a field $L$.
Suppose that
\begin{enumerate}
\item[{\rm (i)}] for any $\sigma\in G$, $\sigma(L)\subset L$;
\item[{\rm (ii)}] the restriction of the action of $G$ to $L$ is faithful;
\item[{\rm (iii)}] for any $\sigma\in G$,
\[
\begin{pmatrix} \sigma(x_1) \\ \sigma(x_2) \\ \vdots \\ \sigma(x_n) \end{pmatrix}
=A(\sigma)\cdot \begin{pmatrix} x_1 \\ x_2 \\ \vdots \\ x_n \end{pmatrix}+B(\sigma)
\]
where $A(\sigma)\in GL_n(L)$ and $B(\sigma)$ is an $n\times 1$ matrix over $L$.
\end{enumerate}

Then there exist elements $z_1,\ldots,z_n\in L(x_1,\ldots,x_n)$ such that
$L(x_1,\ldots,x_n)=L(z_1,\ldots,z_n)$ and $\sigma(z_i)=z_i$ for any $\sigma\in G$,
any $1\le i\le n$.
\end{theorem}

%----------------------------t2.3
\begin{theorem}[{\cite[Theorem 3.1]{AHK}}] \label{t2.3}
Let $L$ be any field, $L(x)$ the rational function field of one variable over $L$,
and $G$ a finite group acting on $L(x)$.
Suppose that, for any $\sigma\in G$, $\sigma(L)\subset L$ and $\sigma(x)=a_\sigma\cdot x+b_\sigma$
where $a_\sigma,b_\sigma\in L$ and $a_\sigma\ne 0$.
Then $L(x)^G=L^G(f)$ for some polynomial $f\in L[x]$.
In fact, if $m=\min \{\deg g(x): g(x)\in L[x]^G\backslash L\}$,
any polynomial $f\in L[x]^G$ with $\deg f=m$ satisfies the property $L(x)^G=L^G(f)$.
\end{theorem}

%-----------------------------t2.4
\begin{theorem}[Hoshi, Kitayama and Yamasaki {\cite[5.4]{HKY}}] \label{t2.4}
Let $k$ be a field with $\fn{char}k\ne 2$, $\varepsilon \in \{1,-1\}$ and $a,b\in k\backslash\{0\}$.
Let $G=\langle\sigma,\tau\rangle$ act on $k(x,y,z)$ by $k$-automorphisms defined by
\begin{align*}
\sigma &: x\mapsto a/x,~ y\mapsto a/y,~ z\mapsto \varepsilon z, \\
\tau &: x\mapsto y \mapsto x, ~ z\mapsto b/z.
\end{align*}
Then $k(x,y,z)^G$ is rational over $k$.
\end{theorem}

%-----------------------------t2.7
\begin{theorem}\label{t2.7}
Let $k$ be a field with $\fn{char}k\ne 2$, $\alpha$ be a non-zero
element in $k$. Let $G$ be a finite group acting on
$k(x_1,x_2,x_3)$ by monomial $k$-automorphisms with coefficients
in $\langle \alpha \rangle$. Assume that $\sqrt{-1}$ belongs to
$k$. Then $k(x_1,x_2,x_3)^G$ is rational over $k$.
\end{theorem}

\begin{proof}
Recall the definition of monomial actions with coefficients in
$\langle \alpha \rangle$ in the last paragraph of Section 1.

For a monomial action of $G$ on $k(x_1,x_2,x_3)$, define the group
homomorphism $\rho_{\underline{x}}:G\to GL_3(\bm{Z})$ by
$\rho_{\underline{x}}(\sigma)=(m_{ij})_{1\le i,j\le 3}$ if
$\sigma\cdot x_j$ is given by $\sigma\cdot
x_j=\lambda_j(\sigma)\cdot \prod_{1\le i\le 3}x_i^{m_{ij}}$ where
$m_{ij}\in\bm{Z}$ and $\lambda_j(\sigma)\in k\backslash \{0\}$
(see \cite[Definition 2.7]{KPr}).

By \cite[Lemma 2.8]{KPr}, we can find a normal subgroup $H$ such
that {\rm (i)} $k(x_1,x_2,x_3)^H$ $=k(z_1,z_2,z_3)$ with each one
of $z_1,z_2,z_3$ in the form
$x_1^{\lambda_1}x_2^{\lambda_2}x_3^{\lambda_3}$ (where
$\lambda_1,\lambda_2,\lambda_3 \in \bm{Z}$), {\rm (ii)} $G/H$ acts
on $k(z_1,z_2,z_3)$ by monomial $k$-automorphisms, and {\rm (iii)}
$\rho_{\underline{z}}: G/H \to GL_3(\bm{Z})$ is injective.

Clearly $G/H$ acts on $k(z_1,z_2,z_3)$ by monomial actions with
coefficients in $\langle \alpha \rangle$ also. Thus we may apply
the results of \cite{Ya,HKY} to $k(z_1,z_2,z_3)^{G/H}$.

It is known that the fixed field of any $3$-dimensional monomial
action on $k(z_1,z_2,z_3)$ is rational except possibly for nine
cases \cite[Theorem 1.6]{HKY}.

The first eight ``exceptional" cases are studied in \cite{Ya}.
Because of our assumptions that $\sqrt{-1}$ and the monomial
actions have coefficients in $\langle \alpha \rangle$, the
rationality criteria found by Yamasaki are satisfied. For example,
in \cite[Lemma 2]{Ya}, $R(a,b,c)$ is affirmative if $a,b,c \in
\langle \alpha \rangle$; for, if both $a$ and $b$ are in $\langle
\alpha \rangle \setminus \langle \alpha^2 \rangle$, then $ab \in
\langle \alpha^2 \rangle$.

The last ``exceptional" case is studied in \cite[Theorem
1.7]{HKY}, which is rescued by the assumption $\sqrt{-1} \in k$.
\end{proof}

%-----------------------------t2.8
\begin{theorem}\label{t2.8}
Let $k$ be a field with $\fn{char}k\ne 2$, $\zeta_m$ be a
primitive $m$-th root of unity belonging to $k$. Let $G$ be a
finite group acting on $k(x_1,x_2,x_3)$ by monomial
$k$-automorphisms with coefficients in $\langle \zeta_m \rangle$.
Then $k(x_1,x_2,x_3)^G$ is rational over $k$ except for the
following situation : there is a normal subgroup $H$ of $G$ such
that $G/H = <\tau>$ is cyclic of order $4$,
$k(x_1,x_2,x_3)^H=k(z_1,z_2,z_3)$, each one of $z_1,z_2,z_3$ is of
the form $x_1^{\lambda_1}x_2^{\lambda_2}x_3^{\lambda_3}${\rm
(}where $\lambda_1,\lambda_2,\lambda_3 \in \bm{Z}${\rm )}, $-1 \in
\langle \zeta_m \rangle$ and $\tau : z_1 \mapsto z_2 \mapsto z_3
\mapsto -1/(z_1z_2z_3)$.

For the exceptional case, $k(x_1,x_2,x_3)^G$ is rational over $k$
if and only if at least one of $-1, 2, -2$ belongs to
$(k^{\times})^2$.

\end{theorem}

\begin{proof}
The proof is very similar to that of Theorem \ref{t2.7}. By
\cite[Lemma 2.8]{KPr}, find the normal subgroup $H$ of $G$ such
that {\rm (i)} $k(x_1,x_2,x_3)^H=k(z_1,z_2,z_3)$ with each one of
$z_1,z_2,z_3$ in the form
$x_1^{\lambda_1}x_2^{\lambda_2}x_3^{\lambda_3}$ (where
$\lambda_1,\lambda_2,\lambda_3 \in \bm{Z}$), {\rm (ii)} $G/H$ acts
on $k(z_1,z_2,z_3)$ by monomial $k$-automorphisms, and {\rm (iii)}
$\rho_{\underline{z}}: G/H \to GL_3(\bm{Z})$ is injective. Thus we
may apply the results of \cite{Ya,HKY} to $k(z_1,z_2,z_3)^{G/H}$.

Clearly $G/H$ acts on $k(z_1,z_2,z_3)$ by monomial actions with
coefficients in $\langle \zeta_m \rangle$ also.

By \cite[Theorem 1.6]{HKY} it remains to check the ``nine cases''
as in the proof of Theorem \ref{t2.7}. Note that, if $m$ is an odd
integer, then $-\zeta_m$ is a primitive $2m$-th root of unity in
$k$ and $\zeta_m \in (k^{\times})^2$.

We may also assume $\sqrt{-1} \notin k$ and thus $4$ doesn't
divide $m$; otherwise, apply Theorem \ref{t2.7} and the proof is
finished. From now on, assume either $m$ is odd or $m=2s$ for some
odd integer $s$.

Apply the results of \cite{Ya,HKY}. All the cases are rational
except the group $G_{4,2,2}$ in Section 4 of \cite{Ya}. This
group, $G/H$, is a cyclic group of order $4$ and the group action
is given by $\tau : z_1 \mapsto z_2 \mapsto z_3 \mapsto
c/(z_1z_2z_3)$ for some $c \in \langle \zeta_m \rangle$.

If $m$ is odd and $c \in \langle \zeta_m \rangle$, we may write
$c=d^4$ for some $d \in \langle \zeta_m \rangle$ (because $gcd
\{m, 4 \} =1$). Define $w_j=z_j/d$ for $1 \le j \le 3$. We find
that $\tau : w_1 \mapsto w_2 \mapsto w_3 \mapsto 1/(w_1w_2w_3)$.
We get a purely monomial action and
$k(x_1,x_2,x_3)^G=k(w_1,w_2,w_3)^{<\tau>}$ is rational by
\cite{HK2}.

Now consider the case where $m=2s$ for some odd integer $s$.

Since $c \in \langle \zeta_{2s} \rangle \simeq C_2 \times C_s$, we
may write $c=\varepsilon b$ where $\varepsilon \in \{1,-1 \}$ and
$b^s=1$. As $b^s=1$, we may write $b=a^4$ for some $a \in
k^{\times}$. Define $w_j=z_j/a$ for $1 \le j \le 3$. We find that
$\tau : w_1 \mapsto w_2 \mapsto w_3 \mapsto
\varepsilon/(w_1w_2w_3)$.

If $\varepsilon = 1$, we get a purely monomial action and
$k(w_1,w_2,w_3)^{<\tau>}$ is rational by \cite{HK2}. Thus only the
case $\varepsilon = -1$ remains. Note that $-1 \in \langle \zeta_m
\rangle$ in the present situation.

\bigskip
Now we will solve the ``exceptional" case whether
$k(x_1,x_2,x_3)^G=k(w_1,w_2,w_3)^{<\tau>}$ is rational over $k$
with $\tau : w_1 \mapsto w_2 \mapsto w_3 \mapsto -1/(w_1w_2w_3)$.

Assume that $k(w_1,w_2,w_3)^{<\tau>}$ is rational. The desired
conclusion ``at least one of $-1, 2, -2$ belongs to
$(k^{\times})^2$", which is equivalent to ``$k(\zeta_8)$ is cyclic
over $k$", is valid when $k$ is a finite field. Hence we may
assume that $k$ is an infinite field.

Consider the fixed field $k(C_8)$ where $k(C_8)=k(y_i: 1 \le i \le
8)^{<\tau>}$ with $\tau : y_1 \mapsto y_2 \mapsto \cdots \mapsto
y_8 \mapsto y_1$. We will show that $k(C_8)$ is rational over
$k(w_1,w_2,w_3)^{<\tau>}$. In fact, define $u_i=y_i + y_{i+4}$ and
$v_i=y_i - y_{i+4}$ for $1 \le i \le 4$. Then $k(y_i: 1 \le i \le
8)=k(u_i,v_i : 1 \le i \le 4)$ and $\tau : u_1 \mapsto u_2 \mapsto
u_3 \mapsto u_4 \mapsto u_1$, $\tau : v_1 \mapsto v_2 \mapsto v_3
\mapsto v_4 \mapsto -v_1$. By Theorem \ref{t2.2}, $k(u_i,v_i : 1
\le i \le 4)^{<\tau>}=k(v_i : 1 \le i \le
4)^{<\tau>}(Y_1,Y_2,Y_3,Y_4)$ for some $Y_j$ where $\tau
(Y_j)=Y_j$ with $1 \le j \le 4$. Define $t_i =v_i/v_{i+1}$ for $1
\le i \le 3$. Then $\tau : t_1 \mapsto t_2 \mapsto t_3 \mapsto
-1/(t_1t_2t_3)$ (the same as the action on $w_1,w_2,w_3$!) and
$k(v_i : 1 \le i \le 4)^{<\tau>}=k(t_1,t_2,t_3)^{<\tau>}(f)$ for
some $f$ with $\tau(f)=f$ by Theorem \ref{t2.3}. Hence $k(C_8)$ is
rational over $k(w_1,w_2,w_3)^{<\tau>}$.

Since $k(w_1,w_2,w_3)^{<\tau>}$ is $k$-rational by assumption, it
is retract $k$-rational \cite[Sa2, Ka7]{Sa} (it is the occasion of
retract rationality that we require the base field $k$ is
infinite). Because the retract rationality is preserved among
stably isomorphic fields \cite[Proposition 3.6; Ka7, Lemma
3.4]{Sa2}, we find that $k(C_8)$ is also retract $k$-rational. By
\cite[Theorem 4.12; Ka7, Theorem 2.9]{Sa2}, we conclude that
$k(\zeta_8)$ is cyclic over $k$. Done.

\bigskip
Now assume that at least one of $-1, 2, -2$ belongs to
$(k^{\times})^2$. We will prove that $k(w_1,w_2,w_3)^{<\tau>}$ is
rational over $k$.

By \cite[Theorem 1.8; Ya, Theorem 3]{Ka7},
$k(w_1,w_2,w_3)^{<\tau>}$ is rational if and only if $-1 \in
(k^{\times})^2$ or $1 \in 4(k^{\times})^4$. The latter condition
is nothing but the condition ``at least one of $-1, 2, -2$ belongs
to $(k^{\times})^2$". Hence the result.
\end{proof}

%-----------------------------t2.5
\begin{theorem}[Hajja \cite{Ha}] \label{t2.5}
Let $G$ be a finite group acting on the rational function field
$k(x,y)$ be monomial $k$-automorphisms. Then $k(x,y)^G$ is
rational over $k$.
\end{theorem}

%-----------------------------t2.6
\begin{theorem}[Kang and Plans {\cite[Theorem 1.3]{KP}}] \label{t2.6}
Let $k$ be any field,
$G_1$ and $G_2$ two finite groups.
If both $k(G_1)$ and $k(G_2)$ are rational over $k$,
then so is $k(G_1\times G_2)$ over $k$.
\end{theorem}

%-------------------------------------------------S3
\section{The proof of Theorem \ref{t1.4}}

We will prove Theorem \ref{t1.4} in this section.

By the discussion of Section 1, it suffices to consider those
groups $G$ in Theorem \ref{t2.1} (with $n\ge 5$) under the
assumptions of \eqref{eq1.1}, i.e. $\fn{char}k\ne 2$ and
$\zeta_{2^{n-3}} \in k$. These assumptions will remain in force
throughout this section.

Write $\zeta=\zeta_{2^{n-3}}\in k$ from now on.
Since $n\ge 5$, $\zeta^{2^{n-5}} \in k$ and $\zeta^{2^{n-5}}$ is a primitive 4-th root of unity.
We write $\zeta^{2^{n-5}}=\sqrt{-1}$.

\bigskip

\begin{Case}{1}  $G=G_1$ where $G_1$ is the group in Theorem \ref{t2.1}.  \end{Case}

$G$ is a metacyclic group.
But we cannot apply Theorem \ref{t1.5} because $\zeta_{2^{n-2}}\notin k$.

Let $V$ be a $k$-vector space whose dual space $V^*$ is defined as $V^*=\bigoplus_{g\in G} k\cdot x(g)$
and $h\cdot x(g)=x(hg)$ for any $g,h\in G$.
Note that $k(G)=k(x(g):g\in G)^G=k(V)^G$.
We will find a faithful $G$-subspace $W$ of $V^*$.

Note that $\langle \sigma^2,\tau\rangle$ is an abelian subgroup of $G$ and $\fn{ord}(\sigma^2)=2^{n-3}$.
Define
\begin{equation} \label{eq3.1}
\begin{aligned}
X &= \sum_{0\le i\le 2^{n-3}-1} \zeta^{-i} \left[ x(\sigma^{2i})+x(\sigma^{2i}\tau)
+ x(\sigma^{2i}\tau^2)+x(\sigma^{2i}\tau^3)\right], \\
Y &= \sum_{0\le i\le 2^{n-3}-1 \atop 0\le j\le 3} \left(\sqrt{-1}\right)^{-j} x(\sigma^{2i}\tau^j).
\end{aligned}
\end{equation}

We find that
\begin{align*}
\sigma^2 &: X \mapsto \zeta X,~ Y\mapsto Y, \\
\tau &: X\mapsto X, ~ Y\mapsto \sqrt{-1} Y.
\end{align*}

Define $x_0=X$, $x_1=\sigma X$, $y_0=Y$, $y_1=\sigma Y$.
The actions of $\sigma$, $\tau$ are given by
\begin{align*}
\sigma &: x_0\mapsto x_1\mapsto \zeta x_0,~ y_0\mapsto y_1\mapsto y_0, \\
\tau &: x_0\mapsto x_0,~ x_1\mapsto -x_1,~ y_0\mapsto\sqrt{-1} y_0,~ y_1\mapsto \sqrt{-1} y_1.
\end{align*}

It follows that $W=k\cdot x_0 \oplus k\cdot x_1\oplus k\cdot y_0\oplus k\cdot y_1$ is a faithful $G$-subspace of $V^*$.
By Theorem \ref{t2.2}, $k(G)$ is rational over $k(x_0,x_1,y_0,y_1)^G$.
It remains to show that $k(x_0,x_1,y_0,y_1)^{\langle \sigma,\tau\rangle}$ is rational over $k$.

Define $z_1=x_1/x_0$, $z_2=y_1/y_0$.
Then $k(x_0,x_1,y_0,y_1)=k(z_1,z_2,x_0,y_0)$ and
\begin{align*}
\sigma &: x_0 \mapsto z_1x_0,~ y_0\mapsto z_2y_0,~z_1\mapsto \zeta/z_1,~ z_2\mapsto 1/z_2, \\
\tau &: x_0\mapsto x_0, ~ y_0\mapsto \sqrt{-1} y_0,~ z_1\mapsto
-z_1,~ z_2\mapsto z_2.
\end{align*}

By Theorem \ref{t2.3},
$k(z_1,z_2,x_0,y_0)^{\langle \sigma,\tau\rangle}=k(z_1,z_2)^{\langle \sigma,\tau\rangle}(z_3,z_4)$ for some
$z_3$, $z_4$ with $\sigma(z_j)=\tau(z_j)=z_j$ for $j=3,4$.

The actions of $\sigma$ and $\tau$ on $z_1$, $z_2$ are monomial automorphisms.
By Theorem \ref{t2.5}, $k(z_1,z_2)^{\langle \sigma,\tau\rangle}$ is rational.
Thus $k(x_0,x_1,y_0,y_1)^{\langle \sigma,\tau\rangle}$ is also rational over $k$.

\bigskip

\begin{Case}{2}  $G=G_2$, $G_3$, $G_{10}$ or $G_{11}$.  \end{Case}

These four groups are direct products of subgroups $\langle
\sigma,\tau\rangle$ and $\langle\lambda\rangle$. We may apply
Theorem \ref{t1.6} to study $k(G)$ since $H:=\langle
\sigma,\tau\rangle$ is a group of order $2^{n-1}$,
$\fn{ord}(\sigma)=2^{n-2}$ and $\zeta_{2^{n-3}}\in k$. By Theorem
\ref{t1.3} we find that $k(H)$ is rational over $k$.

\bigskip

\begin{Case}{3}  $G=G_4$.  \end{Case}

As in the proof of Case 1. $G=G_1$,
we will find a faithful $G$-subspace $W$ in $V^*=\bigoplus_{g\in G} k\cdot x(g)$.
The construction of $W$ is similar to that in Case 1,
but some modification should be made.

Although $\langle \sigma^2,\tau\rangle$ is an abelian subgroup of $G$,
we will consider $\langle\sigma^2\rangle$ instead.
Explicitly, define
\[
X=\sum_{0\le i\le 2^{n-3}-1} \zeta^{-i} \left[x(\sigma^{2i})+x(\sigma^{2i}\tau)\right].
\]

It follows that $\sigma^2(X)=\zeta X$ and $\tau(X)=X$.

Define $x_0=X$, $x_1=\sigma X$, $x_2=\lambda X$, $x_3=\lambda\sigma X$.
We find that
\begin{align*}
\sigma &: x_0\mapsto x_1\mapsto \zeta x_0,~ x_2\mapsto x_3\mapsto \zeta x_2, \\
\tau &: x_0\mapsto x_0,~ x_1\mapsto x_1, ~ x_2\mapsto -x_2, ~ x_3\mapsto -x_3, \\
\lambda &: x_0\mapsto x_2\mapsto x_0,~ x_1\mapsto x_3\mapsto x_1.
\end{align*}

Note that $G$ acts faithfully on $k(x_i:0\le i\le 3)$.
Hence $k(G)$ is rational over $k(x_i:0\le i\le 3)^G$ by Theorem \ref{t2.2}.

Define $y_0=x_0^{2^{n-3}}$, $y_1=x_1/x_0$, $y_2=x_2/x_1$, $y_3=x_3/x_2$.
Then $k(x_i: 0\le i\le 3)^{\langle\sigma^2\rangle}=k(y_i:0\le i\le 3)$ and
\begin{align*}
\sigma &: y_0\mapsto y_1^{2^{n-3}} y_0,~ y_1\mapsto \zeta/y_1,~ y_2\mapsto \zeta^{-1}y_1y_2y_3,~ y_3\mapsto \zeta/y_3, \\
\tau &: y_0\mapsto y_0,~ y_1\mapsto y_1,~ y_2\mapsto -y_2,~ y_3\mapsto y_3, \\
\lambda &: y_0\mapsto y_1^{2^{n-3}}y_2^{2^{n-3}}y_0,~ y_1\mapsto y_3\mapsto y_1,~ y_2\mapsto 1/(y_1y_2y_3).
\end{align*}

By Theorem \ref{t2.3}, we find that $k(y_i:0\le i\le 3)^{\langle\sigma,\tau,\lambda\rangle}
=k(y_i:1\le i\le 3)^{\langle\sigma,\tau,\lambda\rangle}(y_4)$ for some $y_4$ with $\sigma(y_4)=\tau(y_4)=\lambda(y_4)=y_4$.

It is clear that $k(y_i:1\le i\le 3)^{\langle\tau\rangle}=k(y_1,y_2^2,y_3)$.

Define $z_1=y_1$, $z_2=y_3$, $z_3=y_1y_3y_2^2$.
Then $k(y_1,y_2^2,y_3)=k(z_i:1\le i\le 3)$ and
\begin{align*}
\sigma &: z_1 \mapsto \zeta/z_1,~ z_2\mapsto \zeta/z_2,~ z_3\mapsto z_3, \\
\lambda &: z_1\mapsto z_2\mapsto z_1, ~ z_3\mapsto 1/z_3.
\end{align*}

By Theorem \ref{t2.4}, $k(z_i:1\le i\le 3)^{\langle\sigma,\lambda\rangle}$ is rational over $k$.

\bigskip

\begin{Case}{4}  $G=G_5$.  \end{Case}

The proof is similar to Case 3. $G=G_4$.
We define $X$ such that $\sigma^2(X)=\zeta X$,
$\lambda(X)=X$ (note that in the present case we require $\lambda(X)=X$ instead of $\tau(X)=X$).

Define $x_0=X$, $x_1=\sigma X$, $x_2=\tau X$, $x_3=\tau\sigma X$.
It follows that
\begin{align*}
\sigma &: x_0\mapsto x_1 \mapsto \zeta x_0,~ x_2\mapsto x_3\mapsto \zeta x_2, \\
\tau &: x_0\mapsto x_2\mapsto x_0, ~ x_1\mapsto x_3\mapsto x_1, \\
\lambda &: x_0\mapsto x_0,~ x_1\mapsto x_3\mapsto x_1,~ x_2\mapsto x_2.
\end{align*}

It follows that $G$ acts faithfully on $k(x_i:0\le i\le 3)$.
By Theorem \ref{t2.2} it suffices to show that $k(x_i:0\le i\le 3)^G$ is rational over $k$.

Define $y_0=x_0-x_2$, $y_1=x_1-x_3$, $y_2=x_0+x_2$, $y_3=x_1+x_3$.
It follows that $k(x_i:0\le i\le 3)=k(y_0:0\le i\le 3)$ and
\begin{align*}
\sigma &: y_0\mapsto y_1\mapsto \zeta y_0,~ y_2\mapsto y_3\mapsto \zeta y_2, \\
\tau &: y_0\mapsto -y_0,~ y_1\mapsto -y_1,~ y_2\mapsto y_2, ~ y_3\mapsto y_3, \\
\lambda &: y_0\mapsto y_0,~ y_1\mapsto -y_1,~ y_2\mapsto y_2,~ y_3\mapsto y_3.
\end{align*}

By Theorem \ref{t2.2} $k(y_i:0\le i\le 3)^G= k(y_0,y_1)^G (y_4,y_5)$ for some $y_4$,
$y_5$ with $g(y_4)=y_4$, $g(y_5)=y_5$ for any $g\in G$.
Note the the actions of $G$ on $y_0$, $y_1$ are monomial automorphisms.
By Theorem \ref{t2.5} $k(y_0,y_1)^G$ is rational over $k$.

\bigskip

\begin{Case}{5}  $G=G_6, G_7$.  \end{Case}

Consider the case $G=G_6$ first.

Note that $\langle \sigma^2,\tau^2\rangle$ is an abelian subgroup of $G$.
As in the proof of Case 1. $G=G_1$ we define $X$ and $Y$ in $V^*=\bigoplus_{g\in G} k\cdot x(g)$ by
\begin{equation} \label{eq3.2}
\begin{aligned}
X &= \sum_{0\le i\le 2^{n-3}-1} \zeta^{-i} \left[x(\sigma^{2i})+x(\sigma^{2i}\tau^2)\right], \\
Y &= \sum_{0\le i\le 2^{n-3}-1 \atop 0\le j\le 3} \left(\sqrt{-1}\right)^{-j} x(\sigma^{2i}\tau^j).
\end{aligned}
\end{equation}

It follows that $\sigma^2(X)=\zeta X$, $\tau^2(X)=X$, $\sigma^2(Y)=Y$, $\tau(Y)=\sqrt{-1} Y$.

Define $x_0=X$, $x_1=\sigma X$, $x_2=\tau X$, $x_3=\tau\sigma X$, $y_0=Y$, $y_1=\sigma Y$.
We get
\begin{align*}
\sigma &: x_0\mapsto x_1 \mapsto \zeta x_0,~ x_2\mapsto \zeta^{-1} x_3,~ x_3\mapsto x_2,~ y_0\mapsto y_1\mapsto y_0, \\
\tau &: x_0\mapsto x_2\mapsto x_0,~ x_1\mapsto x_3\mapsto x_1,~ y_0\mapsto \sqrt{-1} y_0,~ y_1\mapsto \sqrt{-1}y_1.
\end{align*}

Note that $G$ acts faithfully on $k(x_i,y_0,y_1:0\le i\le 3)$.
We will show that $k(x_i, y_0, y_1:0\le i\le 3)^{\langle \sigma,\tau \rangle}$ is rational over $k$.

Define $y_2=y_1/y_0$.
It follows that $\sigma(y_2)=1/y_2$, $\sigma(y_0)=y_2y_0$, $\tau(y_2)=y_2$, $\tau(y_0)=\sqrt{-1}y_0$.
By Theorem \ref{t2.3} $k(x_i,y_0,y_1: 0\le i\le 3)^{\langle \sigma,\tau \rangle}=k(x_i,y_2,y_0:0\le i\le 3)^{\langle \sigma,\tau \rangle}
=k(x_i,y_2:0\le i\le 3)^{\langle \sigma,\tau \rangle}(y_3)$ for some $y_3$ with $\sigma(y_3)=\tau(y_3)=y_3$.

Define $y_4=(1-y_2)/(1+y_2)$. Then $\sigma(y_4)=-y_4$,
$\tau(y_4)=y_4$. By Theorem \ref{t2.3} $k(x_i,y_2: 0\le i\le
3)^{\langle \sigma,\tau \rangle}=k(x_i:0\le i\le 3)^{\langle
\sigma,\tau \rangle}(y_5)$ for some $y_5$ with
$\sigma(y_5)=\tau(y_5)=y_5$.

Define $z_0=x_0$, $z_1=x_1/x_0$, $z_2=x_3/x_2$, $z_3=x_2/x_1$.
We find that
\begin{align*}
\sigma &: z_0\mapsto z_1z_0,~ z_1\mapsto \zeta/z_1,~ z_2\mapsto \zeta/z_2,~ z_3\mapsto \zeta^{-2} z_1z_2z_3, \\
\tau &: z_0\mapsto z_1z_3z_0,~ z_1\mapsto z_2\mapsto z_1,~ z_3\mapsto 1/(z_1z_2z_3).
\end{align*}

By Theorem \ref{t2.3} $k(x_i:0\le i\le 3)^{\langle \sigma,\tau
\rangle}=k(z_i:0\le i\le 3)^{\langle \sigma,\tau \rangle}
=k(z_i:1\le i\le 3)^{\langle \sigma,\tau \rangle}(z_4)$ for some
$z_4$ with $\sigma(z_4)=\tau(z_4)=z_4$.

Define $u_1=z_3^{2^{n-4}}$.
Then $k(z_i:1\le i\le 3)^{\langle \sigma^2\rangle}=k(z_1,z_2,u_1)$ and
\begin{align*}
\sigma &: z_1\mapsto \zeta/z_1,~z_2\mapsto \zeta/z_2,~ u_1\mapsto (z_1z_2)^{2^{n-4}} u_1, \\
\tau &: z_1\mapsto z_2\mapsto z_1,~ u_1\mapsto \bigl((z_1z_2)^{2^{n-4}}\cdot u_1\bigr)^{-1}.
\end{align*}

Define $u_2=(z_1z_2)^{2^{n-5}}u_1$.
Then $k(z_1,z_2,u_1)=k(z_1,z_2,u_2)$ and
\begin{align*}
\sigma &: z_1\mapsto \zeta/z_1,~z_2\mapsto \zeta/z_2,~ u_2\mapsto -u_2, \\
\tau &: z_1\mapsto z_2\mapsto z_1,~ u_2\mapsto 1/u_2.
\end{align*}

By Theorem \ref{t2.4} $k(z_1,z_2,u_2)^{\langle \sigma,\tau \rangle}$ is rational over $k$.
This solves the case $G=G_6$.

When $G=G_7$,
we use the same $X$ and $Y$ in \eqref{eq3.2}.
Define $x_0$, $x_1$, $x_2$, $x_3$, $y_0$, $y_1$ by the same formula.
The proof is almost the same as $G=G_6$. Done.

\bigskip

\begin{Case}{6}  $G=G_8$.  \end{Case}

Note that $\tau^8=1$ and $\sigma\tau^2=\tau^2\sigma$.

Define $X$ and $Y$ in $V^*=\bigoplus_{g\in G} k\cdot x(g)$ by
\[
X=\sum_{0\le i\le 2^{n-3}-1 \atop 0\le j\le 3} \zeta^{-i} x(\sigma^{2i}\tau^{2j}), \quad
Y=\sum_{0\le i\le 2^{n-3}-1 \atop 0\le j\le 3} \left(\sqrt{-1}\right)^{-j} x(\sigma^{2i} \tau^{2j}).
\]

It follows that $\sigma^2(X)=\zeta X$, $\sigma^2(Y)=Y$, $\tau^2(X)=X$, $\tau^2(Y)=\sqrt{-1}Y$.

Define $x_0=X$, $x_1=\sigma X$, $x_2=\tau X$, $x_3=\tau\sigma X$, $y_0=Y$, $y_1=\sigma Y$, $y_2=\tau Y$, $y_3=\tau\sigma Y$.
We find that
\begin{align*}
\sigma &: x_0 \mapsto x_1\mapsto \zeta x_0,~ x_2\mapsto \zeta^{-1}x_3,~ x_3\mapsto x_2,~ y_0\leftrightarrow y_1,~ y_2\leftrightarrow y_3, \\
\tau &: x_0\leftrightarrow x_2,~ x_1\leftrightarrow x_3,~ y_0\mapsto y_2\mapsto \sqrt{-1}y_0,~ y_1\mapsto y_3\mapsto \sqrt{-1} y_1.
\end{align*}

Since $G=\langle \sigma,\tau \rangle$ acts faithfully on $k(x_i,y_i:0\le i\le 3)$,
it remains to show that $k(x_i,y_i: 0\le i\le 3)^{\langle \sigma,\tau \rangle}$ is rational over $k$.

Define $z_i=x_iy_i$ for $0\le i\le 3$.
We get
\begin{equation} \label{eq3.3}
\begin{aligned}
\sigma &: z_0\mapsto z_1\mapsto \zeta z_0,~ z_2\mapsto \zeta^{-1}z_3,~ z_3\mapsto z_2, \\
\tau &: z_0 \mapsto z_2\mapsto \sqrt{-1} z_0,~ z_1\mapsto z_3\mapsto \sqrt{-1} z_1.
\end{aligned}
\end{equation}

Note that $k(x_i,y_i:0\le i\le 3)=k(x_i,z_i:0\le i\le 3)$ and $G$ acts faithfully on $k(z_i:0\le i\le 3)$.
By Theorem \ref{t2.2} $k(x_i,z_i: 0\le i\le 3)^{\langle \sigma,\tau \rangle}
=k(z_i: 0\le i\le 3)^{\langle \sigma,\tau \rangle}$ $(X_0,X_1,X_2,X_3)$
for some $X_i$ ($0\le i\le 3$) with $\sigma (X_i)=\tau(X_i)=X_i$.

Define $u_0=z_0$, $u_1=z_1/z_0$, $u_2=z_3/z_2$, $u_3=z_2/z_1$.
The actions are given by
\begin{equation} \label{eq3.4}
\begin{aligned}
\sigma &: u_0\mapsto u_1u_0,~ u_1\mapsto \zeta/u_1,~ u_2\mapsto \zeta/u_2,~ u_3\mapsto \zeta^{-2}u_1u_2u_3, \\
\tau &: u_0\mapsto u_1u_3u_0,~ u_1\leftrightarrow u_2,~ u_3\mapsto \sqrt{-1}/(u_1u_2u_3).
\end{aligned}
\end{equation}

By Theorem \ref{t2.3} $k(z_i:0\le i\le 3)^{\langle \sigma,\tau \rangle}=k(u_i:0\le i\le 3)^{\langle \sigma,\tau \rangle}
=k(u_i:1\le i\le 3)^{\langle \sigma,\tau \rangle}$ $(u_4)$ for some $u_4$ with $\sigma(u_4)=\tau(u_4)=u_4$.

Define $v_1=u_3^{2^{n-4}}$.
Then $k(u_i:1\le i\le 3)^{\langle \sigma^2 \rangle} =k(u_1,u_2,v_1)$ and $\sigma(v_1)=(u_1u_2)^{2^{n-4}}$ $v_1$,
$\tau(v_1)=\varepsilon \big/ \bigl((u_1u_2)^{2^{n-4}} u_4\bigr)$ where $\varepsilon=1$ if $n\ge 6$,
and $\varepsilon=-1$ if $n=5$.

Define $v_2=(u_1u_2)^{2^{n-5}}v_1$.
Then $\sigma (v_2)=-v_2$, $\tau(v_2)=\varepsilon/v_2$.
Since $k(u_1,u_2,v_1)^{\langle \sigma,\tau \rangle}=k(u_1,u_2,v_2)^{\langle \sigma,\tau \rangle}$ is rational over $k$
by Theorem \ref{t2.4}, the proof is finished.

\bigskip

\begin{Case}{7}  $G=G_9$.  \end{Case}

Note that $\sigma^2\tau=\tau\sigma^2$.

Define $X$ and $Y$ in $V^*=\bigoplus_{g\in G}k\cdot x(g)$ by
\[
X=\sum_{0\le i\le 2^{n-3}-1 \atop 0\le j\le 3} \zeta^{-i} x(\sigma^{2i}\tau^j), \quad
Y=\sum_{0\le i\le 2^{n-3}-1 \atop 0\le j\le 3} \left(\sqrt{-1}\right)^{-j} x(\sigma^{2i} \tau^j).
\]

It follows that $\sigma^2(X)=\zeta X$, $\sigma^2(Y)=Y$, $\tau(X)=X$, $\tau(Y)=\sqrt{-1}Y$.

Define $x_0=X$, $x_1=\sigma X$, $y_0=Y$, $y_1=\sigma Y$.
We get
\begin{align*}
\sigma &: x_0\mapsto x_1\mapsto\zeta x_0,~ y_0\mapsto y_1\mapsto y_0, \\
\tau &: x_0\mapsto x_0,~ x_1\mapsto x_1,~ y_0\mapsto \sqrt{-1} y_0,~ y_1\mapsto -\sqrt{-1} y_1.
\end{align*}

It remains to prove $k(x_0,x_1,y_0,y_1)^{\langle\sigma,\tau\rangle}$ is rational over $k$.
The proof is almost the same as Case 1. $G=G_1$. Done.

\bigskip

\begin{Case}{8}  $G=G_{12}$.  \end{Case}

Define $X\in V^*=\bigoplus_{g\in G} h\cdot x(g)$ by
\[
X=\sum_{0\le i\le 2^{n-3}-1} \zeta^{-i}\left[ x(\sigma^{2i})+x(\sigma^{2i}\tau)\right].
\]
Then $\sigma^2 X=\zeta X$, $\tau X=X$.

Define $x_0=X$, $x_1=\sigma X$, $x_2=\lambda X$, $x_3=\lambda\sigma X$.
We find that
\begin{align*}
\sigma &: x_0 \mapsto x_1\mapsto \zeta x_0,~ x_2\mapsto \zeta^{-1} x_3,~ x_3\mapsto x_2, \\
\tau &: x_0\mapsto x_0,~ x_1\mapsto x_1,~ x_2\mapsto -x_2,~ x_3\mapsto -x_3, \\
\lambda &: x_0\leftrightarrow x_2,~ x_1\leftrightarrow x_3.
\end{align*}

Since $G=\langle\sigma,\tau,\lambda\rangle$ is faithful on $k(x_i:0\le i\le 3)$,
it remains to show that $k(x_i:0\le i\le 3)^{\langle\sigma,\tau,\lambda\rangle}$ is rational over $k$.

Define $y_0=x_0$, $y_1=x_1/x_0$, $y_2=x_3/x_2$, $y_3=x_2/x_1$.
We get
\begin{equation} \label{eq3.5}
\begin{aligned}
\sigma &: y_0\mapsto y_1y_0,~ y_1\mapsto \zeta/y_1,~ y_2\mapsto \zeta/y_2,~ y_3\mapsto \zeta^{-2} y_1y_2y_3, \\
\tau &: y_0\mapsto y_0,~ y_1\mapsto y_1,~ y_2\mapsto y_2,~ y_3\mapsto -y_3, \\
\lambda &: y_0\mapsto y_1y_3y_0,~ y_1\leftrightarrow y_2,~ y_3\mapsto 1/(y_1y_2y_3).
\end{aligned}
\end{equation}

By Theorem \ref{t2.3} $k(y_i:0\le i\le 3)^{\langle\sigma,\tau,\lambda\rangle}=k(y_i:1\le i\le 3)^{\langle\sigma,\tau,\lambda\rangle}(y_4)$
for some $y_4$ with $\sigma(y_4)=\tau(y_4)=\lambda(y_4)=y_4$.

Define $z_1=y_3^2$. Then $k(y_i:1\le i\le 3)^{\langle
\tau\rangle}=k(y_1,y_2,z_1)$ and
$\sigma(z_1)=\zeta^{-4}y_1^2y_2^2z_1$,
$\lambda(z_1)=1/(y_1^2y_2^2z_1)$.

Define $z_2=z_1^{2^{n-5}}$. Then
$k(y_1,y_2,z_1)^{\langle\sigma^2\rangle}=k(y_1,y_2,z_2)$ and
$\sigma(z_2)=(y_1y_2)^{2^{n-4}}z_2$,
$\lambda(z_2)=1/\bigl((y_1y_2)^{2^{n-4}} z_2\bigl)$.

Define $z_3=(y_1y_2)^{2^{n-5}}z_2$. We find that
$k(y_1,y_2,z_2)=k(y_1,y_2,z_3)$ and $\sigma(z_3)=-z_3$,
$\lambda(z_3)=1/z_3$. By Theorem \ref{t2.4},
$k(y_1,y_2,y_3)^{\langle\sigma,\tau\rangle}$ is rational over $k$.
Done.

\bigskip

\begin{Case}{9}  $G=G_{13}$, $G_{14}$.  \end{Case}

We consider the case $G=G_{13}$ only, because the proof for $G=G_{14}$ is almost the same (with the same way of changing the variables).

Define $X$ and $Y$ in $V^*=\bigoplus_{g\in G} k\cdot x(g)$ by
\begin{align*}
X &= \sum_{0\le i\le 2^{n-3}-1} \zeta^{-i} \left[x(\sigma^{2i})+x(\sigma^{2i}\tau)\right], \\
Y &= \sum_{0\le i\le 2^{n-3}-1} x(\sigma^{2i}) - \sum_{0\le i\le 2^{n-3}-1} x(\sigma^{2i}\tau).
\end{align*}

We find that $\sigma^2(X)=\zeta X$, $\sigma^2 (Y)=Y$, $\tau(X)=X$, $\tau(Y)=-Y$.

Define $x_0=X$, $x_1=\sigma X$, $x_2=\lambda X$, $x_3=\lambda\sigma X$, $y_0=Y$, $y_1=\sigma Y$, $y_2=\lambda Y$, $y_3=\lambda\sigma Y$.
It follows that
\begin{align*}
\sigma &: x_0\mapsto x_1\mapsto\zeta x_0,~ x_2\mapsto \zeta^{-1}x_3,~ x_3\mapsto x_2,~ y_0\leftrightarrow y_1,~ y_2\leftrightarrow -y_3, \\
\tau &: x_i\mapsto x_i,~ y_i \mapsto -y_i, \\
\lambda &: x_0\leftrightarrow x_2,~ x_1\leftrightarrow x_3,~ y_0\leftrightarrow y_2,~ y_1\leftrightarrow y_3.
\end{align*}

Note that $G$ acts faithfully on $k(x_i,y_i:0\le i\le 3)$.
Thus it remains to show that $k(x_i,y_i: 0\le i\le 3)^{\langle\sigma,\tau,\lambda\rangle}$ is rational over $k$.

Define $x_4=y_0+y_1, x_5=y_2+y_3, x_6=y_0-y_1, x_7=y_2-y_3$. Then
$k(x_i,y_i: 0\le i\le 3)=k(x_i: 0\le i\le 7)$, and
$\sigma(x_i)=x_i$ for $i=4,7$, $\sigma(x_i)=-x_i$ for $i=5,6$,
$\tau(x_i)=-x_i$ for $4 \leq i \leq 7$, $\lambda:
x_4\leftrightarrow x_5,~ x_6\leftrightarrow x_7$.

Apply Theorem \ref{t2.2} to  $k(x_i: 0\le i\le 7)$. It suffices to
prove that $k(x_i: 0\le i\le
5)^{\langle\sigma,\tau,\lambda\rangle}$ is rational over $k$.

Define $Z=x_5/x_4$. Then $k(x_i: 0\le i\le 5)=k(x_i, Z: 0\le i\le
4)$ and $\sigma (Z)=-Z, \tau(Z)=Z, \lambda(Z)=1/Z$. Apply Theorem
\ref{t2.3} to $k(x_i: 0\le i\le 5)$. It remains to prove that
$k(x_i, Z: 0\le i\le 3)^{\langle\sigma,\tau,\lambda\rangle}$ is
rational over $k$. Note that the action of $\tau$ becomes trivial
on $k(x_i, Z: 0\le i\le 3)$.

Define $u_0=x_0, u_1=x_1/x_0, u_2=x_3/x_2, u_3=x_2/x_1,u_4=Z$. By
Theorem \ref{t2.3} $k(x_i, Z: 0\le i\le
3)^{\langle\sigma,\lambda\rangle}=k(u_i: 1\le i\le
4)^{\langle\sigma,\lambda\rangle}(U)$ for some element $U$ fixed
by the action of $G$. The actions of $\sigma$ and $\lambda$ are
given by
\begin{align*}
\sigma &: u_1\mapsto \zeta/u_1,~ u_2\mapsto \zeta/u_2,~ u_3\mapsto \zeta^{-2} u_1u_2u_3,
~ u_4\mapsto -u_4, \\
\lambda &: u_1\leftrightarrow u_2,~ u_3\mapsto 1/(u_1u_2u_3),
u_4\mapsto 1/u_4.
\end{align*}

Note that $\sigma^2$ fixes $u_1,u_2,u_4$ and
$\sigma^2(u_3)=\zeta^{-2}u_3$. Define $u_5=u_3^{2^{n-4}}$. Then
$k(u_i: 1\le i\le 4)^{\langle\sigma^2\rangle}=k(u_1,u_2,u_4,u_5)$
and $\sigma(u_5)=(u_1u_2)^{2^{n-4}}u_5,
\lambda(u_5)=1/((u_1u_2)^{2^{n-4}}u_5)$.

Define $u_6=(u_1u_2)^{2^{n-5}}u_5$. Then
$k(u_1,u_2,u_4,u_5)=k(u_1,u_2,u_4,u_6)$ and we get
\begin{align*}
\sigma &: u_1\mapsto \zeta/u_1,~ u_2\mapsto \zeta/u_2,~ u_6\mapsto -u_6,~ u_4\mapsto -u_4, \\
\lambda &: ~ u_1\leftrightarrow u_2,~ u_6\mapsto 1/u_6,~
u_4\mapsto 1/u_4.
\end{align*}

Define $u_7=u_4u_6$. Then $\sigma(u_7)=u_7,\lambda(u_7)=1/u_7$.
Define $u_8=(1-u_7)/(1+u_7)$. Then
$\sigma(u_8)=u_8,\lambda(u_8)=-u_8$. Since
$k(u_1,u_2,u_4,u_6)=k(u_1,u_2,u_6,u_8)$, we may apply Theorem
\ref{t2.3}. Thus it suffices to prove that
$k(u_1,u_2,u_6)^{\langle\sigma,\lambda\rangle}$ is rational over
$k$. By Theorem \ref{t2.4}
$k(u_1,u_2,u_6)^{\langle\sigma,\lambda\rangle}$ is rational over
$k$. Done.

\bigskip

\begin{Case}{10}  $G=G_{15}$, $G_{16}$, $G_{17}$, $G_{18}$, $G_{24}$, $G_{25}$.  \end{Case}

These cases were proved in \cite[Section 5]{Ka4}.
Note that in Cases $5\sim 8$ of \cite[Section 5]{Ka4}, only $\zeta_{2^{n-3}}\in k$ was used.
Hence the result.

\bigskip

\begin{Case}{11}  $G=G_{19}$, $G_{20}$.  \end{Case}

We consider the case $G=G_{19}$ only,
because the proof for $G=G_{20}$ is almost the same.

Define $X\in V^*=\bigoplus_{g\in G} k\cdot x(g)$ by
\[
X=\sum_{0\le i\le 2^{n-3}-1} \zeta^{-i} \left[x(\sigma^{2i})+x(\sigma^{2i}\tau^2)\right].
\]

Then $\sigma^2(X)=\zeta X$ and $\tau^2(X)=X$.

Define $x_0=X$, $x_1=\sigma X$, $x_2=\tau X$, $x_3=\tau\sigma X$.
We find that
\begin{align*}
\sigma &: x_0\mapsto x_1 \mapsto \zeta x_0,~ x_2\mapsto \sqrt{-1} x_3,~ x_3\mapsto \sqrt{-1}\zeta x_2, \\
\tau &: x_0\leftrightarrow x_2,~ x_1\mapsto x_3\mapsto -x_1.
\end{align*}

Thus $G$ acts faithfully on $k(x_i:0\le i\le 3)$.
It remains to prove $k(x_i:0\le i\le 3)^{\langle\sigma,\tau\rangle}$ is rational over $k$.

Define $u_0=x_0$, $u_1=x_1/x_0$, $u_2=x_3/x_2$, $u_3=x_2/x_1$.
We find that
\begin{equation} \label{eq3.6}
\begin{aligned}
\sigma &: u_0\mapsto u_1u_0,~ u_1\mapsto \zeta/u_1,~ u_2\mapsto \zeta/u_2,~ u_3\mapsto \sqrt{-1} \zeta^{-1}u_1u_2u_3, \\
\tau &: u_0\mapsto u_1u_3u_0,~ u_1\mapsto u_2\mapsto -u_1,~ u_3\mapsto 1/(u_1u_2u_3).
\end{aligned}
\end{equation}

Compare the formula \eqref{eq3.6} with the formula \eqref{eq3.4}
in the proof of Case 6. $G=G_8$. It is not difficult to see that
the proof is almost the same as that of Case 6. $G=G_8$ (by taking
the fixed field of the subgroup $<\sigma^2>$ first, and then
making similar changes of variables). Done.

\bigskip

\begin{Case}{12} $G=G_{21}$.  \end{Case}

Note that $\tau^8=1$ and $\sigma^2 \tau=\tau\sigma^2$.

Define $X$ and $Y$ in $V^*=\bigoplus_{g\in G} k\cdot x(g)$ by
\[
X=\sum_{0\le i\le 2^{n-3}-1 \atop 0\le j\le 3} \zeta^{-i} x(\sigma^{2i} \tau^{2j}), \quad
Y=\sum_{0\le i\le 2^{n-3}-1 \atop 0\le j\le 2} \left(\sqrt{-1}\right)^{-j} x(\sigma^{2i} \tau^{2j}).
\]

Then $\sigma^2(X)=\zeta X$, $\sigma^2(Y)=Y$, $\tau^2(X)=X$, $\tau^2(Y)=\sqrt{-1}Y$.

Define $x_0=X$, $x_1=\sigma X$, $x_2=\tau X$, $x_3=\tau\sigma X$, $y_0=Y$, $y_1=\sigma Y$, $y_2=\tau Y$, $y_3=\tau \sigma Y$.
We find that
\begin{align*}
\sigma &: x_0\mapsto x_1\mapsto \zeta x_0,~x_2\mapsto x_3 \mapsto \zeta x_2,~ y_0\leftrightarrow y_1,~ y_2\leftrightarrow \sqrt{-1}y_3, \\
\tau &: x_0\leftrightarrow x_2,~ x_1\leftrightarrow x_3,~
y_0\mapsto y_2\mapsto \sqrt{-1}y_0,~ y_1\mapsto y_3\mapsto -
\sqrt{-1} y_1.
\end{align*}

Since $G$ is faithful on $k(x_i,y_i:0\le i\le 3)$,
it remains to show that $k(x_i,y_i:0\le i\le 3)^{\langle\sigma,\tau\rangle}$ is rational over $k$.

Define $z_i=x_iy_i$ for $0\le i\le 3$. It follows that
\begin{equation} \label{eq3.7}
\begin{aligned}
\sigma &: z_0 \mapsto z_1\mapsto \zeta z_0,~ z_2\mapsto \sqrt{-1} z_3,~z_3\mapsto - \sqrt{-1} \zeta z_2, \\
\tau &: z_0\mapsto z_2\mapsto \sqrt{-1}z_0,~ z_1\mapsto z_3\mapsto
- \sqrt{-1}z_1.
\end{aligned}
\end{equation}

Compare the formulae \eqref{eq3.7} and \eqref{eq3.3}. They are
almost the same. Thus it is obvious that $k(x_i,y_i:0\le i\le
3)^{\langle\sigma,\tau\rangle}$ is rational over $k$.

\bigskip

\begin{Case}{13}  $G=G_{22}$, $G_{23}$.  \end{Case}

We consider the case $G=G_{23}$, because the proof for $G=G_{22}$ is almost the same.

Define $X\in V^*=\bigoplus_{g\in G} k\cdot x(g)$ by
\[
X=\sum_{0\le i\le 2^{n-3}-1} \zeta^{-i} \left[x(\sigma^{2i})+x(\sigma^{2i}\tau)\right].
\]

Then $\sigma^2(X)=\zeta X$, $\tau (X)=X$.

Define $x_0=X$, $x_1=\sigma X$, $x_2=\lambda X$, $x_3=\lambda\sigma X$.
We find that
\begin{align*}
\sigma &: x_0\mapsto x_1\mapsto\zeta x_0,~ x_2\mapsto \sqrt{-1}\zeta^{-1}x_3,~ x_3\mapsto \sqrt{-1}x_2, \\
\tau &: x_0\mapsto x_0, x_1\mapsto x_1,~ x_2\mapsto -x_2,~ x_3\mapsto -x_3, \\
\lambda &: x_0\leftrightarrow x_2,~ x_1\leftrightarrow x_3.
\end{align*}

Note that $G$ acts faithfully on $k(x_i:0\le i\le 3)$.
It remains to show that $k(x_i:0\le i\le 3)^{\langle\sigma,\tau,\lambda\rangle}$ is rational over $k$.

Define $y_0=x_0$, $y_1=x_1/x_0$, $y_2=x_3/x_2$, $y_3=x_2/x_1$.
We get
\begin{equation} \label{eq3.8}
\begin{aligned}
\sigma &: y_0\mapsto y_1y_0,~y_1\mapsto \zeta/y_1,~y_2\mapsto \zeta/y_2,~ y_3\mapsto \sqrt{-1} \zeta^{-2} y_1y_2y_3, \\
\tau &: y_0\mapsto y_0,~ y_1\mapsto y_1,~ y_2\mapsto y_2,~ y_3\mapsto -y_3, \\
\lambda &: y_0\mapsto y_1y_3y_0,~ y_1\leftrightarrow y_2,~ y_3\leftrightarrow 1/(y_1y_2y_3).
\end{aligned}
\end{equation}

Compare the formula \eqref{eq3.8} with the formula \eqref{eq3.5}
in the proof of Case 8. $G=G_{12}$. It is not difficult to show
that $k(x_i:0\le i\le 3)^{\langle\sigma,\tau,\lambda\rangle}$ is
rational over $k$ in the present case.

\bigskip

\begin{Case}{14}  $G=G_{26}$.  \end{Case}

Note that $\lambda^4=1$ and $\sigma^2\tau=\tau\sigma^2$.

Define $X\in V^*=\bigoplus_{g\in G} k\cdot x(g)$ by
\[
X=\sum_{0\le i\le 3} \left(\sqrt{-1}\right)^{-i} \left[x(\sigma^{2i})+x(\sigma^{2i}\tau)\right].
\]

Then $\sigma^2(X)=\sqrt{-1}X$, $\tau(X)=X$.

Define $x_0=X$, $x_1=\sigma X$, $x_2=\lambda X$, $x_3=\lambda\sigma X$.
We find that
\begin{align*}
\sigma &: x_0\mapsto x_1\mapsto \sqrt{-1} x_0,~ x_2\mapsto x_3\mapsto -\sqrt{-1}x_2, \\
\tau &: x_0\mapsto x_0,~ x_1\mapsto -x_1,~ x_2\mapsto x_2,~ x_3\mapsto -x_3, \\
\lambda &: x_0\mapsto x_2\mapsto -x_0,~ x_1\mapsto x_3\mapsto -x_1.
\end{align*}

Since $G$ is faithful on $k(x_i:0\le i\le 3)$,
it remains to show that $k(x_i:0\le i\le 3)^{\langle\sigma,\tau,\lambda\rangle}$ is rational over $k$.

Define $y_0=x_0$, $y_1=x_1/x_0$, $y_2=x_3/x_2$, $y_3=x_2/x_1$.
We get
\begin{align*}
\sigma &: y_0\mapsto y_1y_0,~ y_1\mapsto \sqrt{-1}/y_1,~ y_2\mapsto -\sqrt{-1}/y_2,~ y_3\mapsto -\sqrt{-1}y_1y_2y_3, \\
\tau &: y_0\mapsto y_0,~ y_1\mapsto -y_1,~ y_2\mapsto -y_2,~ y_3\mapsto -y_3, \\
\lambda &: y_0\mapsto y_1y_3y_0,~y_1\leftrightarrow y_2,~ y_3\mapsto -1/(y_1y_2y_3).
\end{align*}

By Theorem \ref{t2.3} $k(y_i:0\le i\le 3)^{\langle\sigma,\tau,\lambda\rangle}=k(y_i:1\le i\le 3)^{\langle\sigma,\tau,\lambda\rangle} (y_4)$
for some $y_4$ with $\sigma(y_4)=\tau(y_4)=\lambda(y_4)=y_4$.

Define $v_0=y_3^2$.
Then $k(y_i:1\le i\le 3)^{\langle \sigma^2\rangle}=k(v_0,y_1,y_2)$ and
\[
\sigma(v_0)=-(y_1y_2)^2v_0,\quad \tau(v_0)=v_0,\quad \lambda(v_0)=1/(y_1^2y_2^2v_0).
\]

Define $v_1=y_1y_2$, $v_2=y_1/y_2$.
Then $k(v_0,y_1,y_2)^{\langle\tau\rangle}=k(v_i:0\le i\le 3)$ and
\begin{align*}
\sigma &: v_1\mapsto 1/v_1,~v_2\mapsto -1/v_2,~ v_0\mapsto -v_1^2v_0, \\
\lambda &: v_1\mapsto v_1,~ v_2\mapsto 1/v_2,~ v_0\mapsto 1/(v_1^2v_0).
\end{align*}

Define $u_1=v_1v_0$, $u_2=v_2$, $u_3=(1-v_1)/(1+v_1)$.
Then $k(v_i:0\le i\le 2)=k(u_i:1\le i\le 3)$ and
\begin{align*}
\sigma &: u_1\mapsto -u_1,~ u_2\mapsto -1/u_2,~ u_3\mapsto -u_3, \\
\lambda &: u_1\mapsto 1/u_1,~ u_2\mapsto 1/u_2,~ u_3\mapsto u_3.
\end{align*}

By Theorem \ref{t2.2} $k(u_i:1\le i\le 3)^{\langle\sigma,\tau\rangle}=k(u_1,u_2)^{\langle\sigma,\tau\rangle}(u_4)$ for some $u_4$ with $\sigma(u_4)=\tau(u_4)=u_4$.
By Theorem \ref{t2.5} $k(u_1,u_2)^{\langle\sigma,\tau\rangle}$ is rational over $k$.
Hence $k(u_i:1\le i\le 3)^{\langle\sigma,\tau\rangle}$ is rational over $k$.
\qed

%-------------------------------------------------S4
\section{The proof of Theorem \ref{t1.8}}

%-----------------t4.1
\begin{theorem} \label{t4.1}
Let $G$ be a group of order $4n$ where $n$ is any positive integer.
Assume that $G$ contains a cyclic subgroup of order $n$.
If $k$ is a field satisfying that $\fn{char}k\ne 2$, $\zeta_n \in k$ and $\sqrt{-1}\in k$,
then $k(G)$ is rational over $k$.
\end{theorem}

\begin{proof}
Choose $\sigma\in G$ such that the order of $\sigma$ is $n$ and write $H=\langle \sigma \rangle$.
Let $G=$ $H\cup \tau_1H\cup \tau_2H\cup \tau_3 H$ be a coset decomposition of $G$ with respect to $H$.
The idea of proving the rationality of $k(G)$ over $k$ is very similar to that in Section 3,
but it is unnecessary to specify the group action of $G$ on the faithful representation space.

Step 1. Choose a vector $X$ with $\sigma \cdot X=\zeta X$ where
$\zeta =\zeta_n$ (see, for example, the proof of Case 1 in Section
3). Find the induced representation by defining $x_0=X$,
$x_1=\tau_1\cdot X$, $x_2=\tau_2\cdot X$, $x_3=\tau_3\cdot X$. For
any $\lambda\in G$, any $x_i$ where $0\le i\le 3$, $\lambda \cdot
x_i=\zeta^c x_j$ for some $j \in \{0,1,2,3\}$ and an integer $c$.
For example, consider $\lambda \cdot x_2$. Since $\lambda \tau_2
\in G$, it follows that $\lambda \tau_2$ belongs to one of the
cosets $H$, $\tau_1H$, $\tau_2H$, $\tau_3H$. Suppose $\lambda
\tau_2 \in \tau_3 H$ and $\lambda \tau_2=\tau_3\sigma^c$. Then
$\lambda \cdot x_2=\lambda \cdot (\tau_2\cdot X)=(\lambda
\tau_2)\cdot X=(\tau_3\sigma^c)\cdot X=\zeta^c x_3$.

Step 2. It is not difficult to see that $G$ acts on the field
$k(x_0,x_1,x_2,x_3)$ faithfully. For, if $\lambda\in G$, $\lambda$
belongs to one of the cosets $H$, $\tau_1H$, $\tau_2H$, $\tau_3H$.
Suppose $\lambda \in \tau_1H$ and $\lambda=\tau_1\sigma^a$ for
some integer $a$, then $\lambda\cdot x_0=\tau_1\sigma^a\cdot
x_0=\tau_1\sigma^a\cdot X=\zeta^a \tau_1\cdot X=\zeta^a x_1$. Thus
$\lambda \cdot x_0\ne x_0$.

Similar to the situation in Case 1 of Section 3, we may apply Theorem \ref{t2.2} to show that $k(G)$ is rational over $k(x_0,x_1,x_2,x_3)^G$.
It remains to show that $k(x_0,x_1,x_2,x_3)^G$ is rational over $k$.

Step 3.
Define $y_i=x_i/x_0$ for $1\le i\le 3$.
By Theorem \ref{t2.3}, $k(x_0,x_1,x_2,x_3)^G=k(y_1,y_2,y_3)^G(f)$ for some polynomial $f\in k(y_1,y_2,y_3)[x_0]$.
It remains to show that $k(y_1,y_2,y_3)^G$ is rational over $k$.

Step 4. Note that the action of $G$ on $k(y_1,y_2,y_3)$ is a
monomial action with coefficient in $\langle \zeta \rangle$. Apply
Theorem \ref{t2.7} (we may apply Theorem \ref{t2.8} as well
because $\sqrt{-1}\in k$). We find that $k(y_1,y_2,y_3)^G$ is
rational over $k$.
\end{proof}

\bigskip
\begin{proof}[Proof of Theorem \ref{t1.8}] \, \, Case 1. $n=4t$ for some integer $t$.

Since $n=4t$, the assumption $\zeta_n \in k$ implies $\sqrt{-1}\in
k$ also. Hence we may apply Theorem \ref{t4.1}.

\bigskip
Case 2. $n$ is an odd integer.

The proof is almost the same as that of Theorem \ref{t4.1}, except
that we don't assume that $\sqrt{-1}\in k$. The proof of Theorem
\ref{t4.1} till Step 3 remains valid in the present situation. But
we will apply Theorem \ref{t2.8} this time. Note that the
``exceptional" case will not arise because $n$ is odd and $-1
\notin \langle \zeta_n \rangle$.

\bigskip
Case 3. $n=2m$ where $m$ is an odd integer.

The proof is also the same as in Theorem \ref{t4.1}, i.e. by using
Theorem \ref{t2.8} it remains to consider the situation that there
is a normal subgroup $H$ of $G$ such that $G/H$ is cyclic of order
$4$.

\bigskip
Step 1. Choose $\sigma\in G$ such that the order of $\sigma$ is
$n$. Since $\sigma^2$ is of odd order $m$, it follows that
$\sigma^2 \in H$. We will show that the cyclic subgroup $\langle
\sigma^2 \rangle$ is normal in $G$.

Note that the $2$-Sylow subgroup, say $C$, of $H$ is of order $2$,
hence is cyclic. By Burnside's $p$-normal complement theorem,
there is a normal subgroup $H_0$ of $H$ such that $H_0$ is of
order $m$ and $H=H_0C$ \cite[Corollary 5.14,p.160]{Is}. In
particular, $H$ is a solvable group.

For any $\lambda \in G$, we will show that $\lambda^{-1} \sigma^2
\lambda \in \langle \sigma^2 \rangle$. Since $\lambda^{-1}
\sigma^2 \lambda \in \lambda^{-1} H \lambda = H$, it follows that
all of $H_0, \langle \sigma^2 \rangle, \langle \lambda^{-1}
\sigma^2 \lambda \rangle$ are subgroups of order $m$ in $H$. By
Hall's Theorem, they are conjugate in $H$ \cite[Theorem 3.14]{Is}.
Since $H_0$ is normal in $H$, we find these three subgroups are
equal to each another. Thus $\lambda^{-1} \sigma^2 \lambda \in
\langle \sigma^2 \rangle$.

\bigskip
Step 2. We will show that the $2$-Sylow subgroups of $G$ are
isomorphic to $C_8$ or $C_4 \times C_2$.

Let $P_1$ be a Sylow $2$-subgroup of $G$ containing $C$, which is
generated by an order $2$ element in $H$. Note $P_1$ is of order
$8$ and $G$ is a semi-direct product of $\langle \sigma^2 \rangle$
and $P_1$. It follows that $G/H \simeq P_1/C$. Thus $P_1$ is a
group of order $8$ with a cyclic quotient of order $4$. Hence
$P_1$ is not isomorphic to the dihedral group of order $8$ or the
quaternion group of order $8$. The only possibility is that $P_1$
is an abelian group. In particular, $P_1 \simeq C_8$ or $C_4
\times C_2$.

\bigskip
Step 3. We will show that $\langle \sigma \rangle$ is a normal
subgroup of $G$ and $\sigma^m \in Z(G)$, the center of $G$.

Choose a $2$-Sylow subgroup $P_2$ containing $\sigma^m$. Note that
$G$ is the semi-direct product of $\langle \sigma^2 \rangle$ and
$P_2$. Since $P_2$ is abelian, $\sigma^m$ commutes with every
element of $P_2$ (and also with $\sigma^2$). Hence it belongs to
$Z(G)$.

It follows that the subgroup generated by $\langle \sigma^2
\rangle$ and $\langle \sigma^m \rangle$ is normal in $G$. This
subgroup is nothing but $\langle \sigma \rangle$.

\bigskip
Step 4. By Step 2, $P_2 \simeq C_8$ or $C_4 \times C_2$.

If $P_2 \simeq C_8$, choose a generator $\tau$ of $P_2$. It
follows that $\tau^4=\sigma^m$. By Burnside's Theorem again, we
find that $G=N \rtimes \langle \tau \rangle$ where $N$ is a normal
subgroup of order $m$ \cite[Corollary 5.14,p.160]{Is}. Hence $G
\simeq C_m \rtimes C_8$. Using Hall's Theorem again \cite[Theorem
3.14]{Is}, we find that $N=\langle \sigma^2 \rangle$.

If $P_2 \simeq C_4 \times C_2$, choose generators $\tau, \tau_1$
of $P_2$ so that $\langle \tau \rangle \simeq C_4$ and $\langle
\tau_1 \rangle \simeq C_2$. If $\sigma^m=\tau^2$, then $G/\langle
\sigma \rangle \simeq C_2 \times C_2$. If $\sigma^m=\tau_1$ or
$\tau_1\tau^2$, we may change the generators $\tau, \tau_1$ and
thus we may assume that $\sigma^m=\tau_1$ with $G=\langle \sigma
\rangle \rtimes \langle \tau \rangle$.

\bigskip
Step 5. Write $\zeta = \zeta_n$.

Subcase 1. $P_2=\langle \tau, \tau_1 \rangle$, $\sigma^m=\tau^2$,
and $G/\langle \sigma \rangle  \simeq C_2 \times C_2$.

As before (see the proof of Theorem \ref{t4.1} for details),
choose a vector $X$ such that $\sigma \cdot X=\zeta X$. Define
$u_0=X$, $u_1=\tau\cdot X$, $u_2=\tau_1\cdot X$,
$u_3=\tau\tau_1\cdot X$. We find that $k(G)$ is rational over
$k(u_0,u_1,u_2,u_3)^G$ and $k(u_0,u_1,u_2,u_3)^G$ is rational over
$k(v_1,v_2,v_3)^G$ where $v_i=u_i/u_0$ for $1\le i\le 3$.

Note that $G$ acts on $k(v_1,v_2,v_3)$ by monomial actions with
coefficients in $\langle \zeta \rangle$.

Since $\langle \sigma \rangle$ is normal in $G$, the action of
$\sigma$ on $u_0,u_1,u_2,u_3$ becomes very simple. We find that
$\sigma (u_j) =\zeta^{b_j}u_j$ for some integer $b_j$. Thus
$\sigma (v_j) =\zeta^{a_j}v_j$ for some integer $a_j$. Now apply
\cite[Lemma 2.8]{KPr} (also see the proof of Theorem \ref{t2.7}
and Theorem \ref{t2.8}). There is a normal subgroup $H$ of $G$
such that $k(v_1,v_2,v_3)^H =k(z_1,z_2,z_3)$ with each one of
$z_1,z_2,z_3$ in the form
$v_1^{\lambda_1}v_2^{\lambda_2}v_3^{\lambda_3}$ (where
$\lambda_1,\lambda_2,\lambda_3 \in \bm{Z}$), $G/H$ acts on
$k(z_1,z_2,z_3)$ by monomial $k$-automorphisms, and
$\rho_{\underline{z}}: G/H \to GL_3(\bm{Z})$ is injective.

Note that $\sigma \in Kernel \{ \rho_{\underline{x}}: G \to
GL_3(\bm{Z}) \}$. Thus $\sigma \in H$. Since $G/H$ is a quotient
of $G/\langle \sigma \rangle \simeq C_2 \times C_2$, $G/H$ is not
isomorphic to the cyclic group of order $4$. Now apply Theorem
\ref{t2.8} to $k(z_1,z_2,z_3)^{G/H}$. We find that
$k(z_1,z_2,z_3)^{G/H}$ is rational.

\bigskip
Subcase 2. $G=\langle \sigma \rangle \rtimes \langle \tau \rangle$
where $\langle \tau \rangle \simeq C_4$.

The proof is similar to that of Subcase 1. Choose the same vector
$X$ with $\sigma \cdot X=\zeta X$. Define $u_j=\tau^j\cdot X$ for
$0 \le j \le 3$. The question is reduced to
$k(u_0,u_1,u_2,u_3)^G$, or to $k(v_1,v_2,v_3)^G$ more simply where
$v_i=u_i/u_{j-1}$ for $1\le i\le 3$.

Again we have $\sigma (v_j) =\zeta^{a_j}v_j$ for some integer
$a_j$. Moreover, since $\tau$ permutes $u_0,u_1,u_2,u_3$
cyclically, we find that $\tau : v_1 \rightarrow v_2 \rightarrow
v_3 \rightarrow 1/(v_1v_2v_3)$.

As in Subcase 1, apply \cite[Lemma 2.8]{KPr} and its proof. We
find that $k(v_1,v_2,v_3)^{<\sigma>}$ $=k(z_1,z_2,z_3)$ where each
one of $z_1,z_2,z_3$ is in the form
$v_1^{\lambda_1}v_2^{\lambda_2}v_3^{\lambda_3}$ (with
$\lambda_1,\lambda_2,\lambda_3 \in \bm{Z}$). Note that $\tau$ acts
on $k(z_1,z_2,z_3)$ by purely monomial $k$-automorphisms, because
we have found $\tau : v_1 \mapsto v_2 \mapsto v_3 \mapsto
1/(v_1v_2v_3)$. Thus $k(z_1,z_2,z_3)^{<\tau>}$ is rational by
\cite{HK2}. Done.

\bigskip
Subcase 3. $P_2$ is generated by $\tau$ with $\langle \tau \rangle
\simeq C_8$ and $\tau^4=\sigma^m$.

We will show that $k(G)$ is rational over $k$ if and only if at
least one of $-1, 2, -2$ belongs to $(k^{\times})^2$.

First assume that $k(G)$ is rational. The proof is almost the same
as the proof of the ``exceptional case" in Theorem \ref{t2.8}. The
desired conclusion ``at least one of $-1, 2, -2$ belongs to
$(k^{\times})^2$", which is equivalent to ``$k(\zeta_8)$ is cyclic
over $k$", is valid when $k$ is a finite field. Hence we may
assume that $k$ is an infinite field.

Since $k(G)$ is $k$-rational, it is retract $k$-rational
\cite[Sa2, Ka7]{Sa}. By Saltman's Theorem \cite[Theorem 3.1; Ka7,
Theorem 3.5]{Sa}, we find that $k(C_8)$ is also retract
$k$-rational because $G \simeq C_m \rtimes C_8$. Since $k(C_8)$ is
retract $k$-rational, it follows that $k(\zeta_8)$ is cyclic over
$k$ by \cite[Theorem 4.12; Ka7, Theorem 2.9]{Sa2}. Done.

\bigskip
Now assume that at least one of $-1, 2, -2$ belongs to
$(k^{\times})^2$. We will prove that $k(G)$ is rational.

Similar to Subcase 2, choose the same vector $X$ with $\sigma
\cdot X=\zeta X$. Define $u_j=\tau^j\cdot X$ for $0 \le j \le 3$.
The question is reduced to $k(u_0,u_1,u_2,u_3)^G$, or to
$k(v_1,v_2,v_3)^G$ more simply where $v_i=u_i/u_{j-1}$ for $1\le
i\le 3$.

Again we have $\sigma (v_j) =\zeta^{a_j}v_j$ for some integer
$a_j$. Moreover, since $\tau^4 =\sigma^m$, we find that $\tau :
v_1 \mapsto v_2 \mapsto v_3 \mapsto -1/(v_1v_2v_3)$.

Apply \cite[Lemma 2.8]{KPr} and its proof. We find that
$k(v_1,v_2,v_3)^{<\sigma>}=k(z_1,z_2,z_3)$ where each one of
$z_1,z_2,z_3$ is in the form
$v_1^{\lambda_1}v_2^{\lambda_2}v_3^{\lambda_3}$ (with
$\lambda_1,\lambda_2,\lambda_3 \in \bm{Z}$). Note that $\tau$ acts
on $k(z_1,z_2,z_3)$ by monomial $k$-automorphisms with
coefficients in $\langle -1 \rangle$, because we have found $\tau
: v_1 \mapsto v_2 \mapsto v_3 \mapsto -1/(v_1v_2v_3)$.

Apply Theorem \ref{t2.8} to $k(z_1,z_2,z_3)^{<\tau>}$. We find
that $k(z_1,z_2,z_3)^{<\tau>}$ is rational except for the
situation $\tau : z_1 \mapsto z_2 \mapsto z_3 \mapsto
-1/(z_1z_2z_3)$. But this ``exceptional" case has been solved in
Theorem \ref{t2.8}. Hence the result.
\end{proof}

\bigskip
It is not difficult to adapt the proof of Theorem \ref{t4.1} so as
to find a short proof of Theorem \ref{t1.3} for 2-groups (but not
for the $p$-groups of odd order in \cite{Ka4}).

The idea of the proof in Theorem \ref{t4.1} can be applied to the $M$-groups defined below.

%--------------------------------d4.2
\begin{defn} \label{d4.2}
Let $k$ be any field, $G$ be a finite subgroup of $GL_n(k)$. $G$
is called an $M$-group if every $\sigma\in G$ has exactly one
non-zero entry in each column. An $M$-group $G$ is related to the
monomial representation which is induced from a one-dimension
representation of some subgroup $H$ of $G$ with $[G:H]\le n$.
Since we use the terminology ``monomial actions", we avoid using
the terminology ``monomial representations" and we just call them
$M$-groups.

An $M$-group $G\subset GL_n(k)$ acts naturally on the rational function field $k(x_1,\ldots,x_n)$ as follows.
For any $\sigma\in G\subset GL_n(k)$, define $\sigma\cdot x_1,\ldots,\sigma\cdot x_n$ by
\[
(\sigma\cdot x_1,\ldots,\sigma\cdot x_n)=(x_1,\ldots,x_n)\sigma,
\]
i.e. if $\sigma=(a_{ij})_{1\le i,j\le n} \in GL_n (k)\subset M_n(k)$, $\sigma\cdot x_j=\sum_{1\le i\le n} a_{ij} x_i$.
\end{defn}

%-----------------------------l4.3
\begin{lemma} \label{l4.3}
Let $k$ be any field, $G$ be a finite $M$-group contained in
$GL_n(k)$. Let $G$ act on the rational function field
$k(x_1,\ldots,x_n)$. Then there exist a root of unity $\zeta_m\in
k$, $b_1,\ldots,b_n \in k\backslash \{0\}$ and such that if we
define $y_i=b_ix_i$ for $1\le i\le n$, then, for $\sigma\in G$,
$1\le i\le n$, $\sigma(y_i)=c y_j$ for some $c\in \langle
\zeta_m\rangle$ and some $1\le j\le n$.
\end{lemma}

\begin{proof}
Note that, for any $\sigma \in G$, $\sigma\cdot x_i=b\cdot x_j$
for some $b\in k\backslash \{0\}$, some $1\le j\le n$, i.e. if we
neglect the coefficient $b$ temporarily, $\sigma$ acts on
$\{x_1,x_2,\ldots,x_n\}$ by some permutation. To prove this lemma,
without loss of generality, we may assume that $G$ acts on
$\{x_1,\ldots,x_n\}$ transitively.

For $2\le i\le n$, choose $\tau_i\in G$ such that $\tau_i \cdot x_1=b_i\cdot x_i$ where $b_i \in k\backslash \{0\}$.
Define $\tau_1=1\in G$ and $b_1=1\in k$.
Thus $\tau_i\cdot x_1=b_i\cdot x_i$ for $1\le i\le n$.

Define $y_i=b_ix_i$ for $1\le i\le n$.

Define $H=\{\tau \in G: \tau\cdot x_1=c\cdot x_1$ for some $c\in k\backslash \{0\}\}$, $I=\{b\in k\backslash \{0\}: b=(\tau\cdot x_1)/x_1$ for some $\tau\in H\}$.
$I$ is a subgroup of the multiplicative group $k\backslash \{0\}$.
Hence $I$ is a cyclic group and $I=\langle \zeta_m\rangle$ for some $\zeta_m\in k$.

We will prove that, for any $\sigma \in G$, any $1\le i\le n$, if $\sigma\cdot y_i=c\cdot y_j$, then $c\in \langle \zeta_m\rangle$.

If $\sigma\in G$ and $(\sigma\cdot x_i)/x_j \in k$, then $(\sigma\cdot y_i)/y_j\in k$.
Write $\sigma\cdot y_i=c\cdot y_j$.
Plugging the formulae $\tau_i\cdot x_1=b_i\cdot x_i=y_i$, $\tau_j x_1=y_j$ into $\sigma \cdot y_i=c\cdot y_j$,
we get $\tau_j^{-1} \sigma \tau_i(x_1)=cx_1$, i.e. $\tau_j^{-1} \sigma \tau_i\in H$ and $c\in I=\langle \zeta_m\rangle$.
\end{proof}

%-----------------------------------t4.4
\begin{theorem} \label{t4.4}
Let $k$ be a field with $\fn{char}k\ne 2$ and $\sqrt{-1}\in k$.
Let $G$ be an $M$-group contained in $GL_4(k)$.
Consider the natural action of $G$ on the rational function field $k(x_1,x_2,x_3,x_4)$.
Then the fixed field $k(x_1,x_2,x_3,x_4)^G$ is rational over $k$.
\end{theorem}

\begin{proof}
By Lemma \ref{l4.3}, find $y_1$, $y_2$, $y_3$, $y_4$ and $\zeta_m \in k$ such that $k(x_1,x_2,x_3,x_4)=k(y_1,y_2,y_3,y_4)$ and,
for any $\sigma \in G$, any $1\le i\le 4$, $\sigma (y_i)=c\cdot y_j$ for some $c\in \langle \zeta_m\rangle$, some $1\le j\le 4$.

Define $z_i=y_i/y_4$ for $1\le i\le 3$. Then
$k(x_1,x_2,x_3,x_4)^G=k(y_1,y_2,y_3,y_4)^G=k(z_1,z_2,z_3)^G(f)$ by
Theorem \ref{t2.3}. Note that $G$ acts on $k(z_1,z_2,z_3)$ by
monomial $k$-automorphisms with coefficients in $\langle \zeta_m
\rangle$. Apply Theorem \ref{t2.7}. We find that
$k(z_1,z_2,z_3)^G$ is rational over $k$.
\end{proof}

Remark. If $\sqrt{-1}$ doesn't belong to the base field $k$, it
may happen that the conclusion of Theorem \ref{t4.4} may fail. For
example, consider an action $ \sigma : x_1 \mapsto x_2 \mapsto x_3
\mapsto x_4 \mapsto \ -x_1$; the fixed field $\bm{Q} (x_1, x_2,
x_3, x_4)^{<\sigma>}$ is not rational over $\bm{Q}$ by
\cite[Theorem 1.8]{Ka7}.

\newpage
%----------------------------------------References
\renewcommand{\refname}{\centering{References}}

\end{document}